\documentclass[a4paper,reqno,12pt]{amsart}
\usepackage[english]{babel}
\usepackage{amsmath}
\usepackage{amssymb}
\usepackage{amscd}
\usepackage{amsthm}
\usepackage{hyperref}
\usepackage{euscript}
\usepackage{enumitem}
\usepackage{cite}
\usepackage[all]{xy}
\usepackage{mathdots}
\usepackage{comment}

\newtheorem{theorem}{Theorem}
\newtheorem{proposition}[theorem]{Proposition}
\newtheorem{lemma}[theorem]{Lemma}
\newtheorem{conjecture}[theorem]{Conjecture}
\newtheorem{problem}[theorem]{Problem}
\newtheorem{corollary}[theorem]{Corollary}

\theoremstyle{definition}
\theoremstyle{remark}
\newtheorem{remark}[theorem]{Remark}
\theoremstyle{definition}
\newtheorem{definition}[theorem]{Definition}
\newtheorem{example}[theorem]{Example}
\newtheorem{construction}[theorem]{Construction}

\DeclareMathOperator{\Spec}{Spec}
\DeclareMathOperator{\Hom}{Hom}
\DeclareMathOperator{\Aut}{Aut}
\DeclareMathOperator{\Der}{Der}

\DeclareMathOperator{\rk}{rk}

\DeclareMathOperator{\Ker}{Ker}
\renewcommand{\Im}{\operatorname{\mathrm{Im}}}

\newcommand{\GL}{\operatorname{GL}}
\newcommand{\SL}{\operatorname{SL}}
\newcommand{\SO}{\operatorname{SO}}

\def\GG{{\mathbb G}}

\def\KK{{\mathbb K}}

\def\ZZ{{\mathbb Z}}

\def\QQ{{\mathbb Q}}

\def\AA{{\mathbb A}}

\renewcommand{\ge}{\geqslant}
\renewcommand{\le}{\leqslant}

\newcommand{\diag}{\operatorname{diag}}
\newcommand{\Lie}{\operatorname{Lie}}
\newcommand{\ord}{\operatorname{ord}}
\newcommand{\Mat}{\operatorname{Mat}}
\newcommand{\tr}{\operatorname{tr}}

\makeatletter

\@addtoreset{theorem}{section}
\@namedef{subjclassname@2020}{%
  \textup{2020} Mathematics Subject Classification}

\makeatother

\numberwithin{equation}{section}

\newcounter{num}
\newcommand{\no}{\refstepcounter{num}\arabic{num}}

\begin{document}
\date{}
\title[Root subgroups on affine spherical varieties]{Root subgroups on affine spherical varieties}
\thanks{This research was supported by the Russian Science Foundation, grant no.~19-11-00056}
\author{Ivan Arzhantsev}
\address{%
{\bf Ivan Arzhantsev} \newline HSE University, Faculty of Computer Science, Pokrovsky boulevard 11, Moscow, 109028 Russia}
\email{arjantsev@hse.ru}
\author{Roman Avdeev}
\address{%
{\bf Roman Avdeev} \newline HSE University, Faculty of Computer Science, Pokrovsky boulevard 11, Moscow, 109028 Russia}
\email{suselr@yandex.ru}



\begin{abstract}
Given a connected reductive algebraic group $G$ and a Borel subgroup~$B \subseteq \nobreak G$, we study $B$-normalized one-parameter additive group actions on affine spherical $G$-varieties.
We establish basic properties of such actions and their weights and discuss many examples exhibiting various features.
We propose a construction of such actions that generalizes the well-known construction of normalized one-parameter additive group actions on affine toric varieties.
Using this construction, for every affine horospherical $G$-variety~$X$ we obtain a complete description of all $G$-normalized one-parameter additive group actions on~$X$ and show that the open $G$-orbit in~$X$ can be connected with every $G$-stable prime divisor via a suitable choice of a $B$-normalized one-parameter additive group action.
Finally, when $G$ is of semisimple rank~$1$, we obtain a complete description of all $B$-normalized one-parameter additive group actions on affine spherical $G$-varieties having an open orbit of a maximal torus $T \subseteq B$.
\end{abstract}

\maketitle


\section{Introduction}
\label{sec0}

Let $\KK$ be an algebraically closed field $\KK$ of characteristic zero.
If the additive group $\GG_a := (\KK, + )$ acts nontrivially on an irreducible algebraic variety~$X$, its image $H$ in the automorphism group $\Aut(X)$ is called a \textit{$\GG_a$-subgroup}.
Moreover, if $X$ is equipped with a regular action of a linear algebraic group~$F$ and $H$ is normalized by~$F$, then we call $H$ an \textit{$F$-root subgroup} on~$X$.
In this case, the action of $F$ on $H$ by conjugation is controlled by a character of~$F$, called the \textit{weight} of~$H$.

When $F = T$ is a torus, $T$-root subgroups on $T$-varieties are used to study the automorphism groups and rationality questions \cite{AHHL,AL,DL,L1,L2,Ni}, transitivity properties for automorphism groups \cite{AKZ1,AKZ2}, equivariant group embeddings \cite{AK,AR}, and affine algebraic monoids \cite{ABZ,DZ}.
The simplest in this setting is the classical case of toric varieties, i.e. normal irreducible $T$-varieties containing an open $T$-orbit.
Toric varieties admit a complete combinatorial description in terms of objects of convex geometry called fans \cite{CLS, Fu, Oda}, and the $T$-root subgroups on a given toric $T$-variety are described in the following simple way: every such subgroup is uniquely determined by its weight and the set of all possible weights is the collection of so-called Demazure roots of the associated fan \cite{De, Oda, Cox, L1}.

A natural intention is to extend the above picture to algebraic varieties equipped with an action of an arbitrary connected reductive algebraic group~$G$.
In this setting, a proper generalization of toric varieties is given by \textit{spherical} varieties, i.e. normal irreducible $G$-varieties containing an open orbit of a Borel subgroup $B \subseteq G$.
These varieties possess many remarkable properties and admit a complete combinatorial description in terms of so-called colored fans, which generalize the fans from the toric case; see~\cite{Tim,Kn91}.
The problem of describing all $G$-root subgroups on affine spherical $G$-varieties was raised in the preprint~\cite{LP}.
It is shown there that such a subgroup is uniquely determined by its weight, and the set of weights is described in some particular cases.
However, the set of $G$-root subgroups on an affine spherical $G$-variety seems to be quite restricted; in particular, it is empty whenever $G$ is semisimple.

In this paper, we initiate a systematic study of $B$-root subgroups on affine spherical $G$-varieties.
One of our aims is to demonstrate that $B$-root subgroups are more natural and suitable generalizations of $T$-root subgroups on $T$-varieties.

By now, $B$-root subgroups on affine spherical $G$-varieties have already appeared in the literature.
In~\cite{RvS}, $B$-root subgroups are used to prove that a smooth affine spherical $G$-variety not isomorphic to a torus is uniquely determined by its automorphism group in the category of smooth affine irreducible varieties.
The only $B$-root subgroups appearing there are a central $\GG_a$-subgroup of the unipotent radical $U \subseteq B$ (see Construction~\ref{constr_central_subgroups}) and all its replicas (see Construction~\ref{constr_replicas}).
In~\cite{Avd}, $B$-root subgroups on affine spherical $G$-varieties naturally arise and play an important role in the study of certain combinatorial invariants of spherical subgroups; see~\S\,\ref{subsec_spherical_subgroups} for details.

As was already mentioned above, every $T$-root subgroup on an affine toric $T$-variety $Z$ is uniquely determined by its weight and the set of weights is described in terms of Demazure roots.
In many applications, the following property of $T$-root subgroups on affine toric $T$-varieties is of particular importance:
every $T$-root subgroup on $Z$ moves a unique $T$-stable prime divisor on~$Z$.
Moreover, there is a bijection between the $T$-stable prime divisors on~$Z$ and the equivalence classes of $T$-root subgroups on~$Z$ under which each $T$-stable prime divisor $D \subseteq X$ corresponds to all $T$-root subgroups on~$Z$ that move~$D$.
See~\S\,\ref{sec2} for details on these results.

It turns out that in the general spherical case the situation is much more complicated.
First, examples show that a $B$-root subgroup on an affine spherical $G$-variety $X$ is not uniquely determined by its weight.
Second, $B$-root subgroups on $X$ are naturally divided into two types: vertical and horizontal.
In geometrical terms, a $B$-root subgroup on $X$ is vertical if it preserves the open~$B$-orbit and horizontal otherwise.
We note that all $T$-root subgroups on affine toric $T$-varieties are horizontal in this terminology.
The set of weights of vertical $B$-root subgroups is rather elusive and we cannot say much about it apart from some straightforward observations.
On the other hand, horizontal $B$-root subgroups admit a certain reduction to the toric case, which imposes rigid restrictions on the set of their weights and hence makes it much more observable, though still far from a complete understanding.
Next, it is easy to see that vertical $B$-root subgroups cannot move any $B$-stable prime divisor on~$X$.
On the other hand, we show that every horizontal $B$-root subgroup moves a unique such divisor.
In the general theory of spherical varieties, $B$-stable prime divisors in~$X$ that are not $G$-stable are called \textit{colors} of~$X$.
We prove that every color in $X$ that is not of type~$a$ (see Definition~\ref{def_type_a}) cannot be moved by a $B$-root subgroup.
On the other hand, examples show that some colors of type~$a$ can be moved by such subgroups.
At last, we conjecture that every $G$-stable prime divisor in~$X$ can be moved by an appropriate $B$-root subgroup (see Conjecture~\ref{conj_G-stable}).

The main contribution of our paper is a general construction of horizontal $B$-root subgroups on an affine spherical $G$-variety~$X$; we call such subgroups \textit{standard}.
In a certain sense, this construction is a generalization of that for $T$-root subgroups on affine toric $T$-varieties.
As a first application of standard $B$-root subgroups, we obtain a complete description of the $G$-root subgroups on~$X$ whose weights belong to the lattice of weights of $G$-semiinvariant rational functions on~$X$.
When $X$ has no colors of type~$a$, this is in fact a description of \textit{all} $G$-root subgroups on~$X$.
The other applications of standard $B$-root subgroups concern the case where $X$ is horospherical (see the definition in~\S\,\ref{subsec_aff_horo}).
First, in the horospherical case there are no colors of type~$a$, hence by the above discussion we get a complete description of all $G$-root subgroups on~$X$.
Second, we obtain a complete description of the set of weights of horizontal $B$-root subgroups on~$X$.
Third, we prove the above-mentioned conjecture on moving $G$-stable prime divisors by $B$-root subgroups (see Theorem~\ref{thm_moved_divisors}).

We also study in detail a particular situation where $G$ is of semisimple rank~$1$ (that is, up to a finite covering, $G$ is isomorphic to the direct product of $\SL_2$ with a torus) and $X$ is an affine spherical $G$-variety having an open orbit of a maximal torus $T\subseteq G$.
Then $X$ is an affine toric $T$-variety and the known description of $T$-root subgroups on~$X$ enables us to describe completely all $B$-root subgroups on~$X$, both vertical and horizontal (see Theorem~\ref{thm_SL2_toric}).
In this case, $X$ is automatically horospherical, so our conjecture also holds for~$X$.

This paper is organized as follows.
In~\S\,\ref{sec1} we discuss preliminary notions and results needed in this paper.
In~\S\,\ref{sec2} we present the well-known combinatorial description of $T$-root subgroups on affine toric $T$-varieties via Demazure roots.
In~\S\,\ref{sec3} we gather all the necessary material on spherical varieties.
In~\S\,\ref{sec4} we discuss basic properties of $B$-root subgroups on affine spherical $G$-varieties and present many examples exhibiting various features.
In~\S\,\ref{sec5} we introduce standard $B$-root subgroups and present their applications, including our description of $G$-root subgroups.
In~\S\,\ref{sec6} we work out the case of an affine toric variety acted on by a connected reductive group of semisimple rank~$1$.

\smallskip

While this paper was under review, a proof of our Conjecture~\ref{conj_G-stable} in the general case appeared in the preprint~\cite{AZ}.
That proof is different from our proof in the horospherical case.

\subsection*{Acknowledgements}

The authors thank Vladimir Zhgoon for his interest in this work and useful discussions.
Thanks are also due to the referee for a careful reading of a previous version of this paper and valuable comments.


\section{Preliminaries}
\label{sec1}

In this section we recall basic facts on algebraic transformation groups used in this paper and provide a framework for the study of root subgroups on affine spherical varieties.

\subsection{General notation and conventions}

Throughout this paper, we work over an algebraically closed field~$\KK$ of characteristic zero.
The notation $\KK^\times$ stands for the multiplicative group $(\KK, \times)$.
The additive group $(\KK,+)$ is denoted by $\GG_a$ and regarded as a one-dimensional linear algebraic group.

If $G$ is an algebraic group then $\mathfrak X(G)$ denotes the character group of~$G$ (in additive notation).

If $X$ is an irreducible algebraic variety then $\Aut(X)$ denotes its automorphism group and $\KK[X]$ (resp.~$\KK(X)$) stands for the algebra of regular functions (resp. field of rational functions) on~$X$.
If in addition $X$ is equipped with a regular action of an algebraic group $H$ then $\KK[X]^H$ (resp.~$\KK(X)^H$) denotes the subalgebra (resp. subfield) of $H$-invariant functions in $\KK[X]$ (resp.~$\KK(X)$).

Given a regular action $G \times X \to X$ of an algebraic group $G$ on an algebraic variety~$X$, we say that the image of $G$ in $\Aut(X)$ is an \emph{algebraic} subgroup of~$\Aut(X)$.
If the group $\Aut(X)$ itself has a structure of an algebraic group such that the natural action $\Aut(X)\times X \to X$ is regular, this definition agrees with the standard definition of an algebraic subgroup in an algebraic group.

\subsection{\texorpdfstring{$\GG_a$}{G\_a}-actions and locally nilpotent derivations}
\label{subsec_Ga-actions&LNDs}

Given an algebraic variety~$X$, by a \emph{$\GG_a$-action} on $X$ we mean a regular action $\GG_a\times X\to X$.
If such an action is nontrivial, it defines a nontrivial algebraic subgroup of $\Aut(X)$, which is called a \emph{$\GG_a$-subgroup}.

A derivation $\partial$ of an algebra $A$ is said to be \emph{locally nilpotent} (LND for short) if for every $a\in A$ there exists $k\in\ZZ_{>0}$ such that $\partial^k(a)=0$.
If $A = \KK[X]$ for an affine algebraic variety~$X$, for any LND $\partial$ on $A$ the map
\begin{equation} \label{eqn_phi_d}
\varphi_{\partial} \colon \GG_a\times A\rightarrow A, \quad (s,a) \mapsto \exp(s\partial)(a),
\end{equation}
defines a rational $\GG_a$-algebra structure on~$A$, hence induces a $\GG_a$-action on~$X$.
In fact, by~\cite[Section~1.5]{Fr} any $\GG_a$-action on~$X$ arises this way, which yields

\begin{proposition} \label{prop_Ga-actions}
Given an affine variety~$X$, the map $\partial \mapsto \varphi_\partial$ induces a bijection between the nonzero LNDs on $\KK[X]$ modulo proportionality and the $\GG_a$-subgroups on~$X$.
\end{proposition}

\subsection{Equivalence of \texorpdfstring{$\GG_a$}{G\_a}-subgroups}
\label{subsec_equivalence}

Let $X$ be an irreducible variety.
We say that two algebraic subgroups $H_1$ and $H_2$ in $\Aut(X)$ are \emph{equivalent} if there exists a nonempty open subset $W\subseteq X$ such that
\[
H_1 x\cap W=H_2 x\cap W \quad \text{for all} \quad x\in W.
\]

By Rosenlicht's Theorem (see, e.g., \cite[Theorem~2.3]{PV}), rational invariants separate orbits of general position in $X$. This implies that two subgroups $H_1$ and $H_2$ are equivalent if and only if the fields $\KK(X)^{H_1}$ and $\KK(X)^{H_2}$ of rational invariants coincide.

If $X$ is affine and $H_1, H_2$ are unipotent, the fields of rational invariants are the quotient fields of the algebras of regular invariants $\KK[X]^{H_1}$ and $\KK[X]^{H_2}$, respectively; see \cite[Theorem~3.3]{PV}.
Since $\KK[X]^{H_i}=\KK[X]\cap\KK(X)^{H_i}$ for $i=1,2$, we conclude that $H_1$ and~$H_2$ are equivalent if and only if
$\KK[X]^{H_1}=\KK[X]^{H_2}$.

\begin{proposition} \label{prop_equivalence}
Suppose $X$ is an irreducible affine variety, $H_1, H_2$ are two $\GG_a$-subgroups in $\Aut(X)$, and $\partial_1, \partial_2$ are the corresponding LNDs on $\KK[X]$.
Then the following conditions are equivalent.
\begin{enumerate}[label=\textup{(\arabic*)},ref=\textup{\arabic*}]
\item \label{equiv1}
$H_1$ and $H_2$ are equivalent.
\item \label{equiv2}
$\Ker \partial_1 = \Ker \partial_2$.
\item \label{equiv3}
There are $a_1,a_2\in\KK[X] \setminus \lbrace 0 \rbrace$ such that $a_1\partial_1=a_2\partial_2$.
\end{enumerate}
\end{proposition}

\begin{proof}
Observe that $\KK[X]^{H_i} = \Ker \partial_i$ for $i=1,2$; then the equivalence (\ref{equiv1})$\Leftrightarrow$(\ref{equiv2}) follows from the above discussion.
The equivalence (\ref{equiv2})$\Leftrightarrow$(\ref{equiv3}) is implied by~\cite[Principle~12]{Fr}.
\end{proof}

\subsection{Root subgroups}
\label{subsec_root_subgroups}

Let $X$ be an irreducible variety and let $F$ be a linear algebraic subgroup of~$\Aut(X)$.

\begin{definition}
An \emph{$F$-root subgroup} on~$X$ is a $\GG_a$-subgroup in $\Aut(X)$ normalized by~$F$.
\end{definition}

Given an $F$-root subgroup $H$ on~$X$, the corresponding group homomorphism $\varphi \colon \GG_a \to \Aut(X)$ satisfies
\[
g \varphi(s) g^{-1} = \varphi(\chi(g)s)
\]
for all $g \in F$, $s \in \KK$ and some $\chi \in \mathfrak X(F)$.
The character $\chi = \chi_H$ is called the \textit{weight} of~$H$.

Now assume that $X$ is affine.
An LND $\partial$ on $\KK[X]$ is said to be \textit{$F$-normalized} if
\[
g \cdot \partial (g^{-1} \cdot f) = \chi(g) \partial(f)
\]
for all $g \in F$, $f \in \KK[X]$ and some character $\chi \in \mathfrak X(F)$, which is called the \textit{weight} of~$\partial$.

The next result is a direct consequence of Proposition~\ref{prop_Ga-actions}.

\begin{proposition}
Given an affine variety~$X$, the map $\partial \mapsto \varphi_\partial$ in~\textup{(\ref{eqn_phi_d})} induces a weight-preserving bijection between the nonzero $F$-normalized LNDs on~$\KK[X]$ modulo proportionality and the $F$-root subgroups on~$X$.
\end{proposition}

\begin{remark}
If $X$ is affine and $F = T$ is a torus then the algebra $\KK[X]$ has the natural grading
\[
\KK[X] = \bigoplus \limits_{u \in \mathfrak X(T)} \KK[X]_u
\]
where each homogeneous component $\KK[X]_u \subseteq \KK[X]$ is the $T$-weight subspace of weight~$u$, i.e. $\KK[X]_u = \lbrace f \in \KK[X] \mid t\cdot f = u(t)f \ \text{for all} \ t \in T \rbrace$.
It is easy to see that an LND $\partial$ on~$\KK[X]$ is $T$-normalized if and only if $\partial$ is \textit{homogeneous}, i.e. sends homogeneous elements to homogeneous ones; see~\cite[Section~3.7]{Fr}.
\end{remark}

\subsection{Orbits of root subgroups}
\label{subsec_orbits}

Let $X$ and $F$ be as in~\S\,\ref{subsec_root_subgroups} ($X$ is not necessarily affine) and let $H$ be an $F$-root subgroup on~$X$.

Take an $H$-orbit $Y \subseteq X$ and let $F_Y$ denote the stabilizer in $F$ of the subvariety~$Y$.
Since $H$ is $F$-normalized, $gY$ is again an $H$-orbit in~$X$ for all $g \in F$.
In particular, the stabilizer in~$F$ of a point in $Y$ is contained in $F_Y$.

Assume that $Y$ is not a point.
Then $Y$ is isomorphic to the affine line $\AA^1$, whose automorphisms are well known to have the form $x \mapsto ax+b$ for some $a,b \in \KK$, $a \ne  0$.
Thus we have the following three possibilities.

\smallskip

{\it Case}~\no: \label{case1}
$F_Y$ acts on $Y$ transitively.
In this case, $Y$ is contained in a single $F$-orbit $O$, which is automatically preserved by~$H$.

\smallskip

{\it Case}~\no: \label{case2}
$F_Y$ acts on $Y$ as a one-dimensional torus.
In this case, $Y$ meets precisely two $F$-orbits in $X$, say $O_1$ and $O_2$.
The stabilizers in $F$ of a point in $O_1$ and of a point in $O_2$ differ by a one-dimensional torus.
So, up to renumbering, we have $\dim O_1 = \dim O_2+1$, $O_2$ meets $Y$ at a single point~$P$, and $O_1$ meets $Y$ in the open subset $Y\setminus \{P\}$.
Then the $H$-orbit of any other point in~$O_1$ meets $O_2$ in a single point whose complement is contained in~$O_1$, and we say that the $F$-orbits $O_1$ and $O_2$ are \textit{connected} by~$H$.

\smallskip

{\it Case}~\no. \label{case3}
$F_Y$ acts on $Y$ as a finite group.
In this case, $Y$ meets infinitely many $F$-orbits in~$X$.

\smallskip

Note the following observations:
\begin{itemize}
\item
if $F$ acts on $X$ with finitely many orbits then Case~\ref{case3} is excluded;

\item
if $F_Y$ is a torus then Case~\ref{case1} does not occur.
\end{itemize}

We say that $H$ \textit{moves} an $F$-stable prime divisor $D \subseteq X$ if $D$ is not $H$-stable.

\begin{proposition} \label{prop_open_orbit}
Under the above assumptions suppose that $F$ is connected and acts on~$X$ with an open orbit~$O_F$.
If $H$ does not preserve~$O_F$ then there is exactly one $F$-stable prime divisor on~$X$ moved by~$H$; moreover, this divisor contains an open $F$-orbit.
\end{proposition}

\begin{proof}
Let $Y \subseteq X$ be an $H$-orbit that meets~$O_F$.
Since the intersection $Y \cap O_F$ is open in~$Y$, Case~\ref{case3} does not occur.
Since $O_F$ is not preserved by~$H$, Case~\ref{case1} is excluded.
So Case~\ref{case2} holds for~$Y$ and hence there is an $F$-orbit $O' \subseteq X$ with $\dim O' = \dim O_F - 1$ such that $O_F$ and $O'$ are connected by~$H$.
Now the closure of $O'$ in $X$ is an $F$-stable prime divisor on~$X$ moved by~$H$.
Clearly, all other $F$-stable prime divisors on~$X$ are $H$-stable.
\end{proof}


\section{Demazure roots and root subgroups on affine toric varieties}
\label{sec2}

In this section we present the well-known description of root subgroups on affine toric varieties.

\subsection{Demazure roots of a cone}
\label{subsec_Dem_roots}

Let $M$ be a lattice of finite rank and consider the dual lattice $N = \Hom_\ZZ(M, \ZZ)$ along with the natural pairing $\langle \cdot\,, \cdot \rangle \colon N \times M \to \ZZ$.
Consider also the rational vector spaces $M_\QQ = M \otimes_\ZZ \QQ$ and $N_\QQ = N \otimes_\ZZ \QQ$ and extend the pairing to a bilinear map $\langle \cdot\,, \cdot \rangle \colon N_\QQ \times M_\QQ \to \QQ$.

Let $\mathcal G \subseteq M_\QQ$ be a finitely generated (or, equivalently, polyhedral) convex cone and recall some terminology related to it.
The cone $\mathcal G$ is said to be \textit{strictly convex} if $\mathcal G \cap (-\mathcal G) = \lbrace 0 \rbrace$, i.e. $\mathcal G$ contains no nonzero subspaces of~$M_\QQ$.
The \textit{dimension} of $\mathcal G$ is that of its linear span.
The \textit{dual cone} of $\mathcal G$ is
\[
\mathcal G^{\vee} := \lbrace q\in N_{\QQ} \mid \langle q,v\rangle\ge 0 \ \text{for all} \ v\in\mathcal G \rbrace;
\]
this is a finitely generated convex cone in~$N_\QQ$.
Note that one automatically has
\[
\mathcal G = \lbrace v \in M_{\QQ} \mid \langle q, v \rangle \ge 0 \ \text{for all} \ q \in \mathcal G^\vee \rbrace,
\]
so that the cone $\mathcal G$ is dual to~$\mathcal G^\vee$.
A \textit{face} of $\mathcal G$ is a subset $\mathcal F \subseteq \mathcal G$ of the form
\[
\mathcal F = \lbrace v \in \mathcal G \mid \langle q, v \rangle = 0 \rbrace
\]
for some $q \in \mathcal G^\vee$.
Every face of $\mathcal G$ is itself a finitely generated convex cone.
Faces of $\mathcal G^\vee$ are defined in a similar way.
A face of codimension one is called a \textit{facet}.
A face of dimension one of a strictly convex cone is called a \textit{ray}.

Now put $\mathcal E := \mathcal G^\vee$ and assume in addition that $\mathcal G$ has full dimension, so that $\mathcal E$ is strictly convex.
Let $\mathcal E^1$ be the set of primitive elements $\rho$ of the lattice~$N$ such that $\QQ_{\ge0}\rho$ is a ray of~$\mathcal E$.
For every $\rho \in \mathcal E^1$, let $\mathcal F_\rho = \mathcal G \cap \Ker \rho$ be the facet of $\mathcal G$ defined by~$\rho$.
Note that the map $\rho \mapsto \mathcal F_\rho$ is a bijection between $\mathcal E^1$ and the facets of~$\mathcal G$.
For every $\rho \in \mathcal E^1$, we define
\[
\mathfrak{R}_{\rho}(\mathcal E) := \lbrace e\in M \mid \langle \rho, e \rangle = -1
\ \text{and} \ \langle \rho', e \rangle \ge 0
 \ \text{for all} \ \rho' \in \mathcal E^1 \setminus \lbrace \rho \rbrace \rbrace.
\]
One easily checks that for $\rk M \ge 2$ the set $\mathfrak R_\rho(\mathcal E)$ is infinite for each $\rho \in \mathcal E^1$.
The elements of the set $\mathfrak R(\mathcal E) := \bigsqcup\limits_{\rho \in \mathcal E^1} \mathfrak{R}_{\rho}(\mathcal E)$ are called the \textit{Demazure roots} of the cone~$\mathcal E$.

For future reference, we mention the following consequence of the construction.

\begin{remark} \label{rem_DR_are_in_M}
One has $\mathfrak R(\mathcal E) \subseteq M$.
\end{remark}

Below in this paper we shall need the following simple result.

\begin{lemma} \label{lemma_two_cones}
Let $\mathcal G, \widetilde{\mathcal G}$ be two finitely generated convex cones of full dimension in $M_\QQ$ and let $\mathcal E := \mathcal G^\vee$, $\widetilde{\mathcal E} := \widetilde{\mathcal G}^\vee$ be the respective dual cones.
Suppose that $\mathcal G \subseteq \widetilde{\mathcal G}$ \textup(and hence $\widetilde{\mathcal E} \subseteq \mathcal E$\textup).
Then for every $\rho \in \mathcal E^1 \setminus \widetilde{\mathcal E}$ the set $\mathfrak R_\rho(\mathcal E) \cap \widetilde{\mathcal G}$ contains infinitely many elements.
\end{lemma}

\begin{proof}
As $\QQ_{\ge0}\rho$ is a ray of $\mathcal E$, there exists an element $v \in \mathcal G$ such that $\langle \rho, v \rangle = 0$ and $\langle \rho', v \rangle > 0$ for all $\rho' \in \mathcal E^1 \setminus \lbrace \rho \rbrace$.
It follows that $\langle q, v \rangle > 0$ for all $q \in \widetilde{\mathcal E} \setminus \lbrace 0 \rbrace$.
Rescaling $v$ if necessary we may assume in addition that $v \in M$.
Then for every $e \in \mathfrak R_\rho(\mathcal E)$ all elements of the form $e + k v$ with $k \ge k_0$ yield the desired infinite set of Demazure roots for a sufficiently large~$k_0 \in \ZZ_{>0}$.
\end{proof}

\subsection{Root subgroups on affine toric varieties}
\label{subsec_RS_on_ATV}

Let $T$ be a torus.
For every $u \in \mathfrak X(T)$, let $\chi^u$ be the regular function on $T$ representing the character~$-u$.
Then $\chi^u$ is $T$-semi\-invariant of weight~$u$, $\chi^{u_1} \cdot \chi^{u_2} = \chi^{u_1 + u_2}$ for all $u_1,u_2 \in \mathfrak X(T)$, and there is the decomposition
\[
\KK[T] = \bigoplus \limits_{u \in \mathfrak X(T)} \KK\chi^u.
\]

A $T$-variety $X$ is said to be \textit{toric} if it is irreducible, normal, and has an open $T$-orbit.
It is well known that affine toric $T$-varieties are parametrized by pairs $(M, \mathcal G)$ where $M$ is a sublattice of $\mathfrak X(T)$ and $\mathcal G$ is a finitely generated convex cone of full dimension in $M_\QQ = M \otimes_\ZZ \QQ$.
More precisely, given such a pair $(M, \mathcal G)$, the corresponding affine toric $T$-variety is $X = \Spec A$ where $A = \bigoplus \limits_{u \in M \cap \mathcal G} \KK \chi^u$.
The monoid $\Gamma = M \cap \mathcal G$ is said to be the \textit{weight monoid} of~$X$, it is characterized as the set of $T$-weights of the algebra~$\KK[X]$.
Note that $M = \ZZ \Gamma$ and $\mathcal G = \QQ_{\ge0}\Gamma$, which gives a direct way to recover the pair $(M,\mathcal G)$ from~$X$.

\begin{remark}
In the definition of a toric $T$-variety it is often additionally required that the action of~$T$ be effective.
This corresponds to $M = \mathfrak X(T)$ in our notation.
\end{remark}

Let $X$ be a toric $T$-variety corresponding to a pair $(M,\mathcal G)$ as above.
Put $N = \Hom_\ZZ(M,\ZZ)$ and retain the notation of~\S\,\ref{subsec_Dem_roots}.

Let $\rho \in \mathcal E^1$ and $e \in \mathfrak R_\rho(\mathcal E)$.
One defines a $T$-normalized LND $\partial_e$ of weight~$e$ on~$A$ by the rule
\begin{equation} \label{eqn_T-norm_LND}
\partial_e(\chi^u)=\langle \rho, u\rangle \chi^{u+e}.
\end{equation}
Let $H_e$ denote the $T$-root subgroup on~$X$ corresponding to~$\partial_e$; see Proposition~\ref{prop_Ga-actions}.
It is known from \cite[Theorem~2.7]{L1} that every nonzero $T$-normalized LND on~$\KK[X]$ has the form $c\partial_e$ for some $e \in \mathfrak R(\mathcal E)$ and $c \in \KK^\times$, which yields

\begin{theorem} \label{thm_T-root_subgroups}
The following assertions hold.
\begin{enumerate}[label=\textup{(\alph*)},ref=\textup{\alph*}]
\item \label{thm_T-root_subgroups_a}
The map $e \mapsto \partial_e$ is a bijection between $\mathfrak R(\mathcal E)$ and the nonzero $T$-normalized LNDs on $\KK[X]$ modulo proportionality.
\item \label{thm_T-root_subgroups_b}
The map $e \mapsto H_e$ is a bijection between $\mathfrak R(\mathcal E)$ and the $T$-root subgroups on~$X$.
\end{enumerate}
\end{theorem}

By construction, $\Ker \partial_e = \bigoplus \limits_{u\in\mathcal F_\rho} \KK\chi^u$.
Combining this with Proposition~\ref{prop_equivalence} we find that, given $e_1,e_2 \in \mathfrak R(\mathcal E)$, the $T$-root subgroups $H_{e_1}$ and $H_{e_2}$ are equivalent if and only if the associated elements $\rho_1, \rho_2 \in \mathcal E^1$ coincide.

Recall from \cite[Theorem~3.2.6]{CLS} that $T$-stable prime divisors on~$X$ are in bijection with~$\mathcal E^1$.
Under this bijection, an element $\rho \in \mathcal E^1$ corresponds to a $T$-stable prime divisor $D_\rho$ whose ideal $I(D_\rho)$ in~$\KK[X]$ equals $\bigoplus \limits_{u \in \Gamma \setminus \mathcal F_\rho} \KK\chi^u$.
Now for every $e \in \mathfrak R_\rho(\mathcal E)$ we have $\KK[X] = \Ker \partial_e \oplus I(D_\rho)$, therefore $I(D_\rho)$ is $\partial_e$-unstable and hence $H_e$-unstable.
Thus $H_e$ moves the divisor~$D_{\rho}$.

\smallskip

We come to the following result, which can be extracted from~\cite[\S\,2]{AKZ1}.

\begin{theorem} \label{thm_toric_equivalence}
Let $X$ be an affine toric $T$-variety.
Then equivalence classes of $T$-root subgroups on $X$ are in bijection with $T$-stable prime divisors on~$X$.
More precisely, in the notation introduced above, the equivalence class corresponding to a divisor $D_{\rho}$ consists of all $T$-root subgroups that move~$D_{\rho}$, and these subgroups are defined by all Demazure roots in~$\mathfrak R_\rho(\mathcal E)$.
\end{theorem}

We finish this subsection with an example illustrating all the concepts introduced above.

\begin{example}[The affine space as a toric variety] \label{ex1}
Consider the affine space $X = \KK^n$ equipped with the standard diagonal action of the torus $T = (\KK^{\times})^n$ given by $t\cdot (x_1,\ldots,x_n) = (t_1x_1,\ldots, t_nx_n)$ for all $t = (t_1,\ldots, t_n) \in T$.
Then $X$ is a toric $T$-variety with an effective action of~$T$, so that $M = \mathfrak X(T)$.
In what follows, we identify $M$ with $\ZZ^n$ by choosing the basis $\chi_1,\ldots,\chi_n$ in $\mathfrak X(T)$ where $\chi_i(t) = t_i^{-1}$ for all $i = 1,\ldots,n$ and $t=(t_1,\ldots,t_n) \in T$.
The dual lattice $N$ will be also identified with $\ZZ^n$ via the dot product.
Modulo the above identifications, we have
$M = N = \ZZ^n$, $M_\QQ = N_\QQ = \QQ^n$, $\Gamma = \ZZ_{\ge0}^n$, $\mathcal G = \mathcal E = \QQ^n_{\geqslant0}$, $\mathcal E^1 = \lbrace \rho_1,\ldots, \rho_n \rbrace$ with
$\rho_1=(1,0,\ldots,0), \ldots, \rho_n=(0,\ldots,0,1)$.
Then
\[
\mathfrak{R}_{\rho_i}(\mathcal E) = \{(c_1,\ldots,c_{i-1},-1,c_{i+1},\ldots,c_n) \mid c_j\in\ZZ_{\geqslant0}\}
\]
for all $i = 1,\ldots, n$.
For $n=2$, the picture looks like
\bigskip
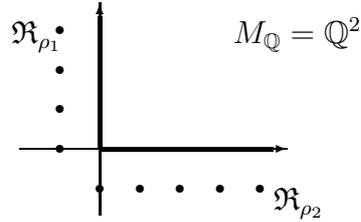
\vspace{0.05cm}
\begin{center}
\begin{picture}(100,75)
\multiput(50,15)(15,0){5}{\circle*{3}}
\multiput(35,30)(0,15){4}{\circle*{3}}
\put(20,30){\vector(1,0){100}} \put(50,5){\vector(0,1){80}}
\put(17,70){$\mathfrak{R}_{\rho_1}$} \put(115,7){$\mathfrak{R}_{\rho_2}$}
\put(100,70){$M_{\mathbb{Q}}=\mathbb{Q}^2$} \linethickness{0.5mm}
\put(50,30){\line(1,0){65}} \put(50,30){\line(0,1){50}}
\end{picture}
\end{center}
Let $x_1,\ldots,x_n$ be the coordinate functions on~$X$, so that $\KK[X] = \KK[x_1,\ldots,x_n]$.
It is easy to see that the $T$-normalized LND on $\KK[X]$ corresponding to a Demazure root
$e=(c_1,\ldots,c_{i-1},-1,c_{i+1},\ldots,c_n)\in \mathfrak{R}_{\rho_i}$ is
\[
\partial_e=x_1^{c_1}\ldots x_{i-1}^{c_{i-1}} x_{i+1}^{c_{i+1}}\ldots x_n^{c_n}\frac{\partial}{\partial x_i}.
\]
This LND gives rise to the $\GG_a$-action
\[
(s,(x_1,\ldots,x_n))\mapsto (x_1,\ldots, x_{i-1}, x_i + sx_1^{c_1}\ldots x_{i-1}^{c_{i-1}} x_{i+1}^{c_{i+1}}\ldots x_n^{c_n}, x_{i+1},\ldots,x_n)
\]
on~$X$.
The corresponding $T$-root subgroup $H_e$ moves the divisor $D_i=\{x_i=0\}$.
\end{example}

\section{Generalities on affine spherical varieties}

\label{sec3}

\subsection{Notation for reductive groups}
\label{subsec_red_groups}

In this subsection we introduce some general notation for reductive groups that will be used in the remaining part of this paper.

In what follows, $G$ denotes a connected reductive algebraic group and $(G,G)$ stands for the derived subgroup of~$G$.
Fix a Borel subgroup $B \subseteq G$ along with a maximal torus $T \subseteq B$.
There is a unique Borel subgroup $B^-$ of $G$ such that $B \cap B^- = T$, it is said to be opposite to~$B$.
Let $U$ (resp.~$U^-$) be the unipotent radical of~$B$ (resp.~$B^-$); both $U$ and $U^-$ are maximal unipotent subgroups of~$G$.
We also put $\mathfrak u = \Lie U$.

We identify the groups $\mathfrak X(B)$ and $\mathfrak X(T)$ via restricting characters from~$B$ to~$T$.
Similarly, we regard $\mathfrak X(G)$ as a subgroup of~$\mathfrak X(T)$.

Let $\Delta \subseteq \mathfrak X(T)$ be the root system of~$G$ with respect to $T$ and let $\Pi \subseteq \Delta$ be the set of simple roots with respect to~$B$.
For every $\alpha \in \Delta$, we let $\alpha^\vee \in \Hom_\ZZ(\mathfrak X(T), \ZZ)$ be the corresponding dual root and let  $U_\alpha \subseteq G$ be the corresponding one-parameter unipotent subgroup.

Let $\Lambda^+ \subseteq \mathfrak X(T)$ be the monoid of dominant weights with respect to~$B$.
Recall that $\Lambda^+$ is in bijection with the (isomorphism classes of) simple finite-dimensional $G$-modules.
Under this bijection, every $\lambda \in \Lambda^+$ corresponds to the simple $G$-module with highest weight~$\lambda$.

\subsection{Spherical varieties and related notions}
\label{subsec_sv&rn}

A $G$-variety $X$ is said to be \textit{spherical} if $X$ is irreducible, normal, and possesses an open $B$-orbit.

Recall that a rational $G$-module $W$ is said to be \textit{multiplicity free} if each simple $G$-module occurs in~$W$ with multiplicity at most one.

\begin{theorem}[{\cite[Theorem~2]{VK}}] \label{thm_VK}
Let $X$ be a normal irreducible $G$-variety.
The following assertions hold.
\begin{enumerate}[label=\textup{(\alph*)},ref=\textup{\alph*}]
\item
If $X$ is spherical then the $G$-module $\KK[X]$ is multiplicity free.

\item
If the $G$-module $\KK[X]$ is multiplicity free and $X$ is quasi-affine then $X$ is spherical.
\end{enumerate}
\end{theorem}

In what follows we let $X$ be a spherical $G$-variety.
The \textit{weight lattice} (resp. \textit{weight monoid}) of $X$ is the set $M = M(X)$ (resp.~$\Gamma = \Gamma(X)$) consisting of weights of $B$-semiinvariant functions in $\KK(X)$ (resp.~$\KK[X]$).
Clearly, $M$ is a sublattice of $\mathfrak X(T)$ and $\Gamma$ is a submonoid of~$M \cap \Lambda^+$.
Thanks to Theorem~\ref{thm_VK}, for every $\lambda \in \Gamma$ there is a unique simple $G$-submodule $\KK[X]_\lambda \subseteq \KK[X]$ with highest weight~$\lambda$, and one has the decomposition
\begin{equation} \label{eqn_decomp}
\KK[X] = \bigoplus \limits_{\lambda \in \Gamma} \KK[X]_\lambda.
\end{equation}

Since $X$ has an open $B$-orbit, for every $\lambda \in M$ there is a unique up to proportionality $B$-semiinvariant function $f_\lambda \in \KK(X)$ of weight~$\lambda$.
Note that for $\lambda \in \Gamma$ the function $f_\lambda$ is a highest-weight vector in $\KK[X]_\lambda$.
Requiring each function~$f_\lambda$ to take the value~$1$ at a fixed point of the open $B$-orbit in~$X$, we may assume $f_\lambda \cdot f_\mu = f_{\lambda+ \mu}$ for all $\lambda, \mu \in \Gamma$.

Let $M_\QQ$, $N$, $N_\QQ$ be as in~\S\,\ref{subsec_Dem_roots}.
Put also $\mathcal G = \QQ_{\ge0}\Gamma \subseteq M_\QQ$ and $\mathcal E = \mathcal G^\vee \subseteq N_\QQ$.
Every discrete $\QQ$-valued valuation $v$ of the field $\KK(X)$ vanishing on~$\KK^\times$ determines an element $\varphi(v) \in N_\QQ$ such that $ \langle \varphi(v), \lambda \rangle = v(f_\lambda)$ for all $\lambda \in M$.
It is known (see~\cite[\S\,7.4]{LV} or~\cite[Corollary~1.8]{Kn91}) that the restriction of the map $v \mapsto \varphi(v)$ to the set of $G$-invariant discrete $\QQ$-valued valuations of $\KK(X)$ vanishing on~$\KK^\times$ is injective; we denote its image by~$\mathcal V = \mathcal V(X)$.
Moreover, $\mathcal V \subseteq N_\QQ$ is a finitely generated convex cone of full dimension containing the image of the antidominant Weyl chamber; see~\cite[Proposition~3.2 and Corollary~4.1,~i)]{BriP}
or~\cite[Corollary~5.3]{Kn91}.
The cone $\mathcal V$ is called the \textit{valuation cone} of~$X$.
Elements of the set
\begin{equation} \label{eqn_sph_roots}
\Sigma = \Sigma(X) := - (\mathcal V^\vee)^1
\end{equation}
are called \textit{spherical roots} of~$X$.
The above discussion implies that every spherical root is a nonnegative linear combination of simple roots.

Since the group $B$ is solvable, the open $B$-orbit in~$X$ is affine, therefore its complement is a union of a finite number of prime divisors (which are automatically $B$-stable).
We note that the open $G$-orbit in $X$ need not be affine and in general its complement may contain irreducible components of codimension~$\ge 2$.

Let $\mathcal D^B = \mathcal D^B(X)$ (resp. $\mathcal D^G = \mathcal D^G(X)$) be the set of $B$-stable (resp. $G$-stable) prime divisors on~$X$ and put $\mathcal D = \mathcal D(X) = \mathcal D^B \setminus \mathcal D^G$.
Elements of $\mathcal D$ are called \textit{colors} of~$X$.
For every $D \in \mathcal D^B$, let $v_D$ be the valuation of $\KK(X)$ defined by~$D$, i.e. $v_D(f) = \ord_D(f)$ for every $f \in \KK(X) \setminus \lbrace 0 \rbrace$.
We define the map
\begin{equation} \label{eqn_varkappa}
\varkappa \colon \mathcal D^B \to N, \quad D \mapsto \varphi(v_D),
\end{equation}
so that $\langle \varkappa(D), \lambda \rangle = \ord_D f_\lambda$ for all $D \in \mathcal D^B$ and $\lambda \in M$.
It follows from the definitions that
\begin{equation} \label{eqn_D^G_in_V}
\varkappa(\mathcal D^G) \subseteq \mathcal V.
\end{equation}
Since $X$ is normal, a~regular function on the open~$B$-orbit in~$X$ is regular on the whole $X$ if and only if it has no poles along all the $B$-stable prime divisors on~$X$, therefore
\begin{equation} \label{eqn_Gamma(X)}
\Gamma = \lbrace \lambda \in M \mid \langle \varkappa(D), \lambda \rangle \ge 0 \ \text{for all} \ D \in \mathcal D^B \rbrace.
\end{equation}
In particular, it follows that the monoid $\Gamma$ is finitely generated and saturated, where the latter means that $\Gamma$ is the intersection of a lattice with a cone.
An equivalent form of~(\ref{eqn_Gamma(X)}) is as follows:
\begin{equation} \label{eqn_E(X)}
\mathcal E = \QQ_{\ge 0} \lbrace \varkappa(D) \mid D \in \mathcal D^B \rbrace.
\end{equation}
In particular, for every element $\rho \in \mathcal E^1$ there is $D \in \mathcal D^B$ such that $\varkappa(D)$ is a positive multiple of~$\rho$.

Let $O$ be the open $G$-orbit in~$X$, which is a homogeneous spherical $G$-variety.
Then
\begin{equation} \label{eqn_Gamma(O)}
\Gamma(O) = \lbrace \lambda \in M \mid \langle \varkappa(D), \lambda \rangle \ge 0 \ \text{for all} \ D \in \mathcal D \rbrace.
\end{equation}
Observe that the algebra $\KK[X]$ is identified with the subalgebra $\bigoplus \limits_{\lambda \in \Gamma} \KK[O]_\lambda$ of~$\KK[O]$.

To finish this subsection, we single out a certain subclass of colors of~$X$.
For every $\alpha \in \Pi$, let $P_\alpha \supseteq B$ be the minimal parabolic subgroup of~$G$ containing~$U_{-\alpha}$ and put $\mathcal D_\alpha := \lbrace D \in \mathcal D \mid D \ \text{is moved by} \ P_\alpha \rbrace$.
Then $\mathcal D = \bigcup \limits_{\alpha \in \Pi} \mathcal D_\alpha$.
The next result is extracted from~\cite[\S\S\,2.7,\,3.4]{Lu97}; see also~\cite[\S\,30.10]{Tim}.

\begin{proposition} \label{prop_types_of_colors}
For every $\alpha \in \Pi$, the following assertions hold.
\begin{enumerate}[label=\textup{(\alph*)},ref=\textup{\alph*}]
\item \label{prop_types_of_colors_a}
$|\mathcal D_\alpha| \le 2$ and the equality is attained if and only if $\alpha \in \Sigma$.

\item \label{prop_types_of_colors_b}
If $|\mathcal D_\alpha| = 2$ and $\mathcal D_\alpha \cap \mathcal D_\beta \ne \varnothing$ for some $\beta \in \Pi$ then $|\mathcal D_\beta| = 2$.

\item \label{prop_types_of_colors_c}
If $|\mathcal D_\alpha| = 1$ and $\mathcal D_\alpha = \lbrace D \rbrace$ then there is $c \in \lbrace 1, \frac12 \rbrace$ such that $\langle \varkappa(D), \lambda \rangle = c\langle \alpha^\vee, \lambda \rangle$ for all $\lambda \in M$.
\end{enumerate}
\end{proposition}

In view of parts~(\ref{prop_types_of_colors_a}),\,(\ref{prop_types_of_colors_b}) of the above proposition, the following notion is well defined.

\begin{definition} \label{def_type_a}
A color $D \in \mathcal D$ is of \textit{type~$a$} if $D \in \mathcal D_\alpha$ for some $\alpha \in \Pi \cap \Sigma$.
\end{definition}

\begin{remark} \label{rem_colors_of_type_a}
It follows from Definition~\ref{def_type_a} and Proposition~\ref{prop_types_of_colors}(\ref{prop_types_of_colors_a}) that $X$ has colors of type~$a$ if and only if $\Pi \cap \Sigma \ne \varnothing$.
\end{remark}

Combining Proposition~\ref{prop_types_of_colors}(\ref{prop_types_of_colors_a},\,\ref{prop_types_of_colors_c}) with~(\ref{eqn_Gamma(O)}) we obtain

\begin{proposition} \label{prop_Gamma(O)}
Suppose that $\Pi \cap \Sigma = \varnothing$.
Then $\Gamma(O) = M \cap \Lambda^+$.
\end{proposition}

\begin{remark}
As follows from the definitions, for given $D \in \mathcal D$ and $\alpha \in \Pi$ the condition $D \in \mathcal D_\alpha$ holds if and only if $D$ is moved by~$U_{-\alpha}$, which induces a $T$-root subgroup on~$X$ that is not a $B$-root subgroup.
If in addition $\alpha \in \Sigma$ then we know from Proposition~\ref{prop_types_of_colors}(\ref{prop_types_of_colors_a}) that $U_{-\alpha}$ moves exactly two colors; compare with Proposition~\ref{prop_open_orbit}.
\end{remark}

\subsection{Affine spherical varieties}
\label{subsec_ASV}

In this subsection we present several notions and results specific to affine spherical varieties.
Until the end of this subsection we assume that $X$ is an affine spherical $G$-variety and retain the notation of \S\,\ref{subsec_sv&rn}.

The following result is well known; see, e.g., \cite[Proposition~5.14]{Tim} for a proof.
It implies that the cone $\mathcal G$ has full dimension in~$M_\QQ$ and hence the cone $\mathcal E$ is strictly convex, which will be important in our subsequent considerations.

\begin{proposition} \label{prop_ASV_WL}
One has $M = \ZZ\Gamma$.
\end{proposition}

The algebra $\KK[X]^U$ decomposes as
\begin{equation} \label{eqn_K[X]^U}
\KK[X]^U = \bigoplus \limits_{\lambda \in \Gamma} \KK f_\lambda.
\end{equation}
By~\cite[Theorem~3.13]{PV}, $\KK[X]^U$ is finitely generated, hence we can consider the variety $Z:=\Spec\KK[X]^U$.
Observe that $Z$ is a toric $T$-variety with weight monoid~$\Gamma$.
The inclusion $\KK[X]^U \subseteq \KK[X]$ gives rise to a dominant $T$-equivariant morphism
\begin{equation} \label{eqn_pi_U}
\pi_U\colon X\to Z.
\end{equation}

\begin{proposition}
A generic fiber of~$\pi_U$ is a single $U$-orbit.
\end{proposition}

\begin{proof}
The discussion in~\S\,\ref{subsec_equivalence} implies that the functions in $\KK[X]^U$ separate generic $U$-orbits in~$X$.
Since all $U$-orbits in~$X$ are closed (see~\cite[\S\,1.3]{PV} or~\cite[Lemma~3.4]{Tim}), this yields the assertion.
\end{proof}

Since $\KK[X]^G = \KK$, it follows that $X$ contains a unique closed $G$-orbit.
Then similarly to~\cite[Lemma~2.4]{Kn91} one proves the following result.

\begin{proposition} \label{prop_rho_G-stable}
The restriction of $\varkappa$ to $\mathcal D^G$ is an injective map to $\mathcal E^1$ and its image is $\lbrace \rho \in \mathcal E^1 \mid \QQ_{\ge0}\rho \cap \varkappa(\mathcal D) = \varnothing \rbrace$.
\end{proposition}

Given $\lambda,\mu, \nu \in \Gamma$ such that the linear span of $\KK[X]_\lambda \cdot \KK[X]_\mu$ contains $\KK[X]_\nu$, the expression $\lambda + \mu - \nu$ is said to be a \textit{tail} of~$X$.
The cone $\mathcal T = \mathcal T(X) \subseteq M_\QQ$ generated by all tails of~$X$ is said to be the \textit{tail cone} of~$X$.

\begin{proposition}[{see~\cite[Lemma~6.6, iii)]{Kn96}}] \label{prop_tail_cone}
Both cones $\mathcal T$ and $\mathcal T(O)$ coincide with $-\mathcal V^\vee$.
\end{proposition}

In view of~(\ref{eqn_sph_roots}) and~(\ref{eqn_D^G_in_V}) we obtain

\begin{corollary} \label{crl_tail_cone}
The following assertions hold.
\begin{enumerate}[label=\textup{(\alph*)},ref=\textup{\alph*}]
\item \label{crl_tail_cone_a}
$\Sigma = \mathcal T^1$.

\item \label{crl_tail_cone_b}
For all $D \in \mathcal D^G$ and $\tau \in \mathcal T$, one has $\langle \varkappa(D), \tau \rangle \le 0$.
\end{enumerate}
\end{corollary}

\subsection{Affine horospherical varieties}
\label{subsec_aff_horo}

An irreducible $G$-variety $X$ is said to be \textit{horospherical} if the stabilizer of a point in general position in~$X$ contains a maximal unipotent subgroup of~$G$.
A normal horospherical $G$-variety is spherical if and only if it contains an open $G$-orbit.
By abuse of terminology, from now on and till the end of this paper all horospherical $G$-varieties are assumed to be spherical.

Affine horospherical varieties were studied in detail in~\cite{VP72} under the name ``$S$-varieties''.
The following well-known result is deduced essentially from~\cite[Theorem~6]{VP72}.

\begin{theorem} \label{thm_horospherical}
Given an affine spherical $G$-variety $X$, the following conditions are equivalent.
\begin{enumerate}[label=\textup{(\arabic*)},ref=\textup{\arabic*}]
\item
$X$ is horospherical.

\item \label{thm_horospherical_2}
$\mathcal T(X) = \lbrace 0 \rbrace$ or, equivalently, $\KK[X]_\lambda \cdot \KK[X]_\mu \subseteq \KK[X]_{\lambda + \mu}$ for all $\lambda, \mu \in \Gamma(X)$.
\end{enumerate}
Moreover, the map $X \mapsto \Gamma(X)$ is a bijection between affine horospherical $G$-varieties, considered up to $G$-equivariant isomorphisms, and submonoids in~$\Lambda^+$ that are finitely generated and saturated.
\end{theorem}

\begin{remark}
Condition~(\ref{thm_horospherical_2}) of Theorem~\ref{thm_horospherical} means that decomposition~\textup{(\ref{eqn_decomp})} is a grading.
\end{remark}

\begin{remark} \label{rem_Sigma_is_empty}
Under the conditions of Theorem~\ref{thm_horospherical}, one has $\Sigma = \varnothing$ by Corollary~\ref{crl_tail_cone}(\ref{crl_tail_cone_a}).
\end{remark}

\begin{remark} \label{rem_AHV_construction}
Given a finitely generated and saturated submonoid $\Gamma \subseteq \Lambda^+$, the affine horospherical $G$-variety $X$ with $\Gamma(X) = \Gamma$ can be constructed in the following simple way; see~\cite[\S\,3]{VP72}.
Choose a finite generating set $\mathrm E$ for $\Gamma$.
For every $\lambda \in \mathrm E$, let $V_\lambda$ be the simple $G$-module with highest weight~$\lambda$ and choose a highest-weight vector $v_\lambda$ in the dual $G$-module~$V_\lambda^*$.
Put
\[
V = \bigoplus \limits_{\lambda \in \mathrm E} V_\lambda^* \ \text{ and } \ v = \sum \limits_{\lambda \in \mathrm E} v_\lambda \in V.
\]
Then $X$ can be realized as the closure of the orbit $G\cdot v$ in~$V$.
\end{remark}

\subsection{Classification results for affine spherical varieties}
\label{subsec_ASV_comb_descr}

The following result was first proved in~\cite[Theorem~1.2]{Lo09}; see also~\cite[Corollary~4.16]{ACF1} for a different proof.

\begin{theorem}
Up to a $G$-equivariant isomorphism, an affine spherical $G$-variety $X$ is uniquely determined by the pair $(\Gamma(X), \Sigma(X))$.
\end{theorem}

Given a finitely generated and saturated submonoid~$\Gamma \subseteq \Lambda^+$, we already know from Theorem~\ref{thm_horospherical} that there exists an affine horospherical $G$-variety $X_0$ with $\Gamma(X_0) = \Gamma$, in which case $\Sigma(X_0) = \varnothing$ by Remark~\ref{rem_Sigma_is_empty}.
A complete description of all possible sets~$\Sigma$ for which there exists an affine spherical $G$-variety $X$ with $\Gamma(X) = \Gamma$ and $\Sigma(X) = \Sigma$ was obtained in~\cite[Theorem~6.9]{ACF2} and \cite[Proposition~2.13]{BvS}; there are always only finitely many such sets~$\Sigma$.
The above-cited sources also explain how to recover $\mathcal D(X)$ as an abstract set equipped with the map $\varkappa \colon \mathcal D(X) \to N$ starting from the pair $(\Gamma(X), \Sigma(X))$.
The set $\mathcal D^G(X)$ along with the map $\varkappa \colon \mathcal D^G(X) \to N$ is then recovered by Proposition~\ref{prop_rho_G-stable}.

\subsection{Notation for various objects assigned to an affine spherical variety}

\label{subsec_ASV_notation}

For convenience of the reader and future reference, in this subsection we list the notation for various objects assigned to an affine spherical $G$-variety~$X$.
This notation will be systematically used in the rest of our paper.

$\Gamma$ is the weight monoid of~$X$

$O$ is the open $G$-orbit in~$X$

$\Gamma(O)$ is the weight monoid of~$O$

$Z = \Spec \KK[X]^U$; this is the affine toric $T$-variety with weight monoid~$\Gamma$

$M$ is the weight lattice of~$X$; $M = \ZZ \Gamma$ by Proposition~\ref{prop_ASV_WL}

$M_\QQ = M \otimes_\ZZ \QQ$ is the rational vector space spanned by~$M$

$N = \Hom_\ZZ(M, \ZZ)$ is the dual lattice of~$M$

$N_\QQ = N \otimes_\ZZ \QQ$ is the dual vector space of~$M_\QQ$

$\iota \colon \Hom_\ZZ(\mathfrak X(T), \ZZ) \to N$ is the restriction map

$\mathcal G = \QQ_{\ge0}\Gamma \subseteq M_\QQ$ is the cone in~$M_\QQ$ generated by~$\Gamma$

$\mathcal E = \mathcal G^\vee \subseteq N_\QQ$ is the cone in~$N_\QQ$ dual to~$\mathcal G$

$\mathcal E^1$ is the set of primitive elements $\rho \in N$ such that $\QQ_{\ge0}\rho$ is a ray of~$\mathcal E$

$\mathcal D^B$ is the set of $B$-stable prime divisors in~$X$

$\mathcal D^G \subseteq \mathcal D^B$ is the set of $G$-stable prime divisors in~$X$

$\mathcal D = \mathcal D^B \setminus \mathcal D^G$ is the set of colors of~$X$

$\varkappa \colon \mathcal D^B \to N$ is the map given by~(\ref{eqn_varkappa})

$\Sigma \subseteq M$ is the set of spherical roots of~$X$

$\mathcal T \subseteq M_\QQ$ is the tail cone of~$X$ and also of~$O$

$\mathcal T^\perp = \lbrace \rho \in N_\QQ \mid \langle \rho, \tau \rangle = 0 \ \text{for all} \ \tau \in \mathcal T \rbrace$

$f_\lambda \in \KK(X)$ is a fixed $B$-semiinvariant function of weight~$\lambda \in M$; we assume $f_\lambda \cdot f_\mu = f_{\lambda+\mu}$ for all $\lambda, \mu \in M$

\section{Basic properties of \texorpdfstring{$B$}{B}-root subgroups}
\label{sec4}

Let $G$ be a connected reductive algebraic group, fix a Borel subgroup $B \subseteq G$, and retain all the notation of~\S\,\ref{subsec_red_groups}.
Throughout this section, unless otherwise specified, $X$ denotes an affine spherical $G$-variety.
For objects related to~$X$, we use the notation listed in~\S\,\ref{subsec_ASV_notation}.

In all examples presented in this section, $X$ is a \textit{spherical $G$-module}, i.e. a finite-dimen\-sional $G$-module that is spherical as a $G$-variety.
A plenty of useful information on spherical $G$-modules, including their classification and combinatorial invariants, can be found in~\cite{Kn98}.
In all our examples where $X = \KK^n$ for some~$n$ we use the following conventions: $G$ is a subgroup of $\GL_n$ acting linearly on~$X$; the group $B$ (resp. $U$, $T$) consists of all upper triangular (resp. upper unitriangular, diagonal) matrices contained in~$G$; $x_i$ is the $i$th coordinate function on~$X$.

\subsection{First properties of \texorpdfstring{$B$}{B}-root subgroups}
\label{subsec_first_properties}

Throughout this subsection, $X$ is an affine irreducible $G$-variety (not necessarily spherical).
Let $H$ be a $B$-root subgroup on $X$ of weight~$\chi_H$ and let $\partial$ be the corresponding $B$-normalized LND of~$\KK[X]$.

\begin{proposition} \label{prop_weight_is_dom}
The following assertions hold.
\begin{enumerate}[label=\textup{(\alph*)},ref=\textup{\alph*}]
\item \label{prop_weight_is_dom_a}
$\chi_H \in \Lambda^+$.
\item \label{prop_weight_is_dom_b}
$H$ is $G$-normalized \textup(and hence a $G$-root subgroup on~$X$\textup) if and only if $\chi_H \in \mathfrak X(G)$.
\end{enumerate}
\end{proposition}

\begin{proof}
Consider the $G$-module $\Der(\KK[X])$ of all derivations of the algebra~$\KK[X]$.
Since the $G$-module $\KK[X]$ is rational (see~\cite[Lemma~1.4]{PV}) and any derivation is uniquely determined by its values on some finite-dimensional $G$-invariant generating subspace in $\KK[X]$, we conclude that the $G$-module $\Der(\KK[X])$ is rational as well.
Now $\partial$ is a $B$-semiinvariant vector in this module, which yields~(\ref{prop_weight_is_dom_a}).
Clearly, $\partial$ is $G$-normalized if and only if it generates a one-dimensional $G$-submodule in $\Der(\KK[X])$.
Since the latter condition is equivalent to $\chi_H \in \mathfrak X(G)$, we get~(\ref{prop_weight_is_dom_b}).
\end{proof}

\begin{proposition} \label{prop_U^-}
For every $G$-stable subspace $V \subseteq \KK[X]$ generating $\KK[X]$ as an algebra, the derivation $\partial$ is uniquely determined by its restriction to $V^{U^-}$.
In particular, $\partial$ is uniquely determined by its restriction to $\KK[X]^{U^-}$ and also to any system of generators of~$\KK[X]^{U^-}$.
\end{proposition}

\begin{proof}
Clearly, $\partial$ is uniquely determined by its restriction to~$V$.
In turn, the action of $\partial$ on $V$ is uniquely determined by that on $V^{U^-}$ because $\partial$ commutes with $\mathfrak u = \Lie U$ and every simple $G$-module is the linear span of elements obtained from a lowest-weight vector via the action of~$\mathfrak u$.
\end{proof}

\subsection{Vertical and horizontal \texorpdfstring{$B$}{B}-root subgroups}

Since the group $U$ has no nontrivial characters, every $B$-root subgroup $H$ on $X$ commutes with~$U$.
In particular, the action of~$H$ on~$\KK[X]$ preserves the subalgebra $\KK[X]^U$ and hence induces an action of $H$ on $Z = \Spec \KK[X]^U$ such that the morphism $\pi_U$ in~(\ref{eqn_pi_U}) is $H$-equivariant.

\begin{definition}
A $B$-root subgroup $H$ on~$X$ is called \textit{vertical} if $H$ acts trivially on $\KK[X]^U$ and \textit{horizontal} otherwise.
\end{definition}

The above notions naturally extend to $B$-normalized LNDs on~$\KK[X]$.
Namely, such an LND is vertical if and only if it vanishes on $\KK[X]^U$ and horizontal otherwise.

Clearly, a $B$-root subgroup $H$ on~$X$ is vertical if and only if it acts trivially on~$Z$ and horizontal otherwise.
It follows that every vertical $B$-root subgroup on~$X$ preserves all fibers of~$\pi_U$ and every horizontal one permutes them, which explains the terminology.
Let $O_B$ be the open $B$-orbit in~$X$.
The following proposition provides a geometrical criterion for a $B$-root subgroup on~$X$ to be vertical or horizontal.

\begin{proposition} \label{prop_open_B-orbit}
Let $H$ be a $B$-root subgroup on~$X$.
\begin{enumerate}[label=\textup{(\alph*)},ref=\textup{\alph*}]
\item \label{prop_open_B-orbit_a}
$H$ is vertical if and only if it preserves $O_B$.
In this case, $H$ moves no $B$-stable prime divisors on~$X$.

\item \label{prop_open_B-orbit_b}
$H$ is horizontal if and only if it does not preserve~$O_B$.
In this case, $H$ moves precisely one $B$-stable prime divisor on~$X$.
\end{enumerate}
\end{proposition}

\begin{proof}
Thanks to Proposition~\ref{prop_open_orbit}, it suffices to prove part~(\ref{prop_open_B-orbit_a}).

Suppose that $H$ is vertical.
Since a generic fiber of $\pi_U$ is a single $U$-orbit, it follows that $H$ preserves generic $U$-orbits in~$X$.
Then generic $H$-orbits in~$X$ are contained in~$O_B$, hence $H$ preserves~$O_B$; see Case~\ref{case1} in~\S\,\ref{subsec_orbits}.

Conversely, suppose that $H$ preserves~$O_B$.
Then every $H$-orbit $Y \subseteq O_B$ fits into Case~\ref{case1} of~\S\,\ref{subsec_orbits} and hence coincides with an orbit of a $\GG_a$-subgroup in $B_Y$.
Such a $\GG_a$-subgroup is contained in $U$, therefore $Y$ is contained in a single $U$-orbit.
It follows that all functions in $\KK[X]^U$ are constant along generic $H$-orbits, hence $H$ acts trivially on $\KK[X]^U$ and $H$ is vertical.
\end{proof}

\begin{remark}
All root subgroups on affine toric varieties are horizontal.
\end{remark}

\begin{proposition} \label{prop_B-rs_equiv}
Let $H_1,H_2$ be two equivalent $B$-root subgroups on~$X$.
\begin{enumerate}[label=\textup{(\alph*)},ref=\textup{\alph*}]
\item \label{prop_B-rs_equiv_a}
$H_1,H_2$ are either both vertical or both horizontal.
\item \label{prop_B-rs_equiv_b}
If $H_1,H_2$ are horizontal then they move the same $B$-stable prime divisor on~$X$.
\end{enumerate}
\end{proposition}

\begin{proof}
If one of $H_1,H_2$ is vertical then so is the other one thanks to condition~(\ref{equiv2}) of Proposition~\ref{prop_equivalence}, which yields~(\ref{prop_B-rs_equiv_a}).
Now suppose $H_1$ is horizontal and moves a divisor $D \in \mathcal D^B$.
Let $W \subseteq X$ be a nonempty open subset such that $H_1x \cap W = H_2x \cap W$ for all $x \in W$.
Since $H_i$-orbits in~$X$ are closed for $i=1,2$ (see~\cite[\S\,1.3]{PV} or~\cite[Lemma~3.4]{Tim}), it follows that $H_1x = H_2x$ for all $x \in W$.
Then for every $x \in W \cap O_B$ the orbit $H_1x$ meets~$D$ (see Case~\ref{case2} of \S\,\ref{subsec_orbits}) and coincides with $H_2x$, therefore $D$ is $H_2$-unstable.
Applying Proposition~\ref{prop_open_B-orbit}(\ref{prop_open_B-orbit_b}) completes the proof of~(\ref{prop_B-rs_equiv_b}).
\end{proof}

\subsection{Simplest constructions of \texorpdfstring{$B$}{B}-root subgroups and/or \texorpdfstring{$B$}{B}-normalized LNDs}

\label{subsec_simplest_constructions}

In this subsection we discuss several basic constructions providing a number of nontrivial examples of $B$-root subgroups on affine spherical $G$-varieties or, equivalently, $B$-normalized LNDs on their algebras of regular functions.

\begin{construction}[Central subgroups] \label{constr_central_subgroups}
Suppose that the derived subgroup $(G,G)$ is nontrivial and acts on~$X$ with a finite kernel (for example, effectively).
Then the image in $\Aut(X)$ of every $T$-normalized central $\GG_a$-subgroup of $U$ is a vertical $B$-root subgroup on~$X$.
Such a central $\GG_a$-subgroup in~$U$ always exists; for simple $(G,G)$ it is unique and corresponds to the highest root of~$B$.
\end{construction}

\begin{construction}[Replicas] \label{constr_replicas}
Suppose $\partial$ is a $B$-normalized LND on~$\KK[X]$ of weight~$\chi$ and $\lambda \in \Gamma$ is such that $f_\lambda \in \Ker \partial$.
Then $f_\lambda\partial$ is a $B$-normalized LND on~$\KK[X]$ of weight $\chi + \lambda$.
Moreover, $\partial$ and $f_\lambda \partial$ are equivalent by Proposition~\ref{prop_equivalence}.
\end{construction}

\begin{construction}[Group extensions] \label{constr_G_1}
Suppose that the action of $G$ on $X$ extends to an action of a bigger connected reductive algebraic group $G_1 \supseteq G$ and the derived subgroup of~$G_1$ is nontrivial and acts on~$X$ with a finite kernel.
If $B_1 \supseteq T_1$ are a Borel subgroup and a maximal torus of $G_1$ satisfying $B_1 \supseteq B$, $T_1 \supseteq T$ and $U_1$ is the unipotent radical of~$B_1$ then the image in $\Aut(X)$ of every $T_1$-normalized central $\GG_a$-subgroup of $U_1$ is a $B$-root subgroup on~$X$.
\end{construction}

The following example shows that $B$-root subgroups arising in Construction~\ref{constr_G_1} may be not only vertical but also horizontal.

\begin{example}[Horizontal $B$-root subgroups arising from group extensions] \label{ex_hor_from_ge}
Take $X = \KK^6$ and let $G \subseteq \GL_6$ be the subgroup of all matrices of the form $\begin{pmatrix} sA & 0 \\ 0 & sA^\sharp \end{pmatrix}$ where $s \in \KK^\times$, $A \in \GL_3$, and $A^\sharp$ stands for the transpose of $A^{-1}$ with respect to the antidiagonal.
Then
\[
T = \lbrace \diag(st_1,st_2,st_3, st_3^{-1},st_2^{-1},st_1^{-1}) \mid s,t_1,t_2,t_3 \in\nobreak \KK^\times \rbrace
\]
and a basis of $\mathfrak X(T)$ is given by the characters $\chi,\chi_1,\chi_2,\chi_3$ where $\chi(t) = s^2$, $\chi_1(t) = st_1$, $\chi_2(t) = st_2$, $\chi_3(t) = st_3$ for all $t = \diag(st_1,st_2,st_3, st_3^{-1},st_2^{-1},st_1^{-1}) \in T$.
The algebra $\KK[X]^U$ is freely generated by the functions
\[
f_1 = x_3, \ \: f_2 = x_6, \ \: \text{and} \ \: f_3 = x_1x_6 + x_2x_5 + x_3x_4
\]
of weights $-\chi_3$, $\chi_1 - \chi$, and $-\chi$, respectively.
Now let $G_1 \subseteq \GL_6$ be the subgroup consisting of all matrices of the form $sA_1$ with $s \in \KK^\times$ and $A_1 \in \SO_6$ where $\SO_6$ preserves the quadratic form~$f_3$.
Put also $G_2 = \GL_6$ and note that $G \subseteq G_1 \subseteq G_2$.
For $i = 1,2$ we choose $B_i$ (resp. $U_i$, $T_i$) to be the subgroup of all upper-triangular (resp. upper unitriangular, diagonal) matrices in~$G_i$.
Let $H,H_1,H_2 \subseteq \GL_6$ be the one-parameter unipotent subgroups consisting of all matrices of the form
\[
\begin{pmatrix}
1 & 0 & x & 0 & 0 & 0\\
0 & 1 & 0 & 0 & 0 & 0\\
0 & 0 & 1 & 0 & 0 & 0\\
0 & 0 & 0 & 1 & 0 & -x\\
0 & 0 & 0 & 0 & 1 & 0\\
0 & 0 & 0 & 0 & 0 & 1
\end{pmatrix},
\quad
\begin{pmatrix}
1 & 0 & 0 & 0 & x & 0\\
0 & 1 & 0 & 0 & 0 & -x\\
0 & 0 & 1 & 0 & 0 & 0\\
0 & 0 & 0 & 1 & 0 & 0\\
0 & 0 & 0 & 0 & 1 & 0\\
0 & 0 & 0 & 0 & 0 & 1
\end{pmatrix},
\quad
\begin{pmatrix}
1 & 0 & 0 & 0 & 0 & x\\
0 & 1 & 0 & 0 & 0 & 0\\
0 & 0 & 1 & 0 & 0 & 0\\
0 & 0 & 0 & 1 & 0 & 0\\
0 & 0 & 0 & 0 & 1 & 0\\
0 & 0 & 0 & 0 & 0 & 1
\end{pmatrix},
\]
respectively.
Then $H,H_1,H_2$ are central subgroups in $U,U_1,U_2$, respectively, normalized by~$B$.
It is easy to check that all the functions $f_1,f_2,f_3$ are $H$-invariant and $H_1$-invariant and $f_3$ is not $H_2$-invariant, thus the $B$-root subgroups induced by $H$ and $H_1$ are vertical and that induced by $H_2$ is horizontal.
Moreover, a simple analysis of orbits in general position for $H$ and $H_1$ shows that $H$ and $H_1$ are not equivalent.
\end{example}

\begin{remark} \label{rem_X'}
Even for simple $G$ a vertical $B$-root subgroup on~$X$ need not be equivalent to a central $B$-root subgroup.
Indeed, in the situation of Example~\ref{ex_hor_from_ge}, taking $X' = \lbrace f_3 =\nobreak 0 \rbrace \subseteq\nobreak X$, $G' = (G,G) \simeq \SL_3$, and $B' = B \cap G'$, we would find that $X'$ is a spherical $G'$-variety and $H,H_1$ still induce nonequivalent vertical $B'$-root subgroups on~$X'$.
\end{remark}

\begin{construction}[Partial derivations] \label{constr_SM}
Let $V$ be a spherical $G$-module, choose a simple $G$-submodule $V_0 \subseteq V$ with highest weight~$\lambda$, and fix a $G$-module decomposition $V = V_0 \oplus W$.
Fix a lowest-weight vector $f \in V_0^*$ and regard it as a coordinate function on~$V$ via the chain $V_0^* \subseteq V_0^* \oplus W^* \simeq V^* \subseteq \KK[V]$. (The asterisk always denotes the dual $G$-module.)
Let $R \subseteq V^*$ be the $T$-invariant complement of the line $\KK f$.
Then the derivation $\partial/\partial f$ with respect to the decomposition $V^* = \KK f \oplus R$ is locally nilpotent and $B$-normalized of weight~$\lambda$.
\end{construction}

\begin{remark}
The above construction can yield both vertical and horizontal $B$-root subgroups.
For instance, in the situation of Example~\ref{ex_hor_from_ge}, $\partial/\partial x_1$ is horizontal (resp. horizontal, vertical) as a $B$-normalized (resp. $B_1$-normalized, $B_2$-normalized) LND on~$X$ regarded as a spherical $G$-module (resp. $G_1$-module, $G_2$-module).
\end{remark}

\subsection{\texorpdfstring{$B$}{B}-root subgroups and the open \texorpdfstring{$G$}{G}-orbit}

While a vertical $B$-root subgroup preserves the open $B$-orbit in~$X$ (see Proposition~\ref{prop_open_B-orbit}(\ref{prop_open_B-orbit_a})), it may not preserve the open $G$-orbit.

\begin{example}[A vertical $B$-root subgroup not preserving the open $G$-orbit]
Take $X =\nobreak \KK^2$, $G = \SL_2$ and consider the $B$-normalized LND $\partial/\partial x_1$ on~$X$ (see Construction~\ref{constr_SM}).
Then the corresponding $B$-root subgroup $H$ on~$X$ acts as $(s,(x_1,x_2)) \mapsto (x_1+s,x_2)$.
Since the points $(1,0)$ and $(0,0)$ lie in the same $H$-orbit, the open $G$-orbit in $X$ is not $H$-stable.
Note that the image of~$U$ in $\Aut(X)$, which acts as $(s, (x_1,x_2)) \mapsto (x_1+sx_2,x_2)$, is a replica of~$H$ (see Construction~\ref{constr_replicas}).
\end{example}

The above example is a manifestation of the following general result.

\begin{proposition}
Let $G$ be a connected semisimple algebraic group acting on a smooth affine variety $Y$ \textup(not necessarily spherical\textup) with an open orbit.
Assume that the $G$-action on $Y$ is not transitive.
Then there is a $B$-root subgroup $H$ in $\Aut(Y)$ that does not preserve the open $G$-orbit.
\end{proposition}

\begin{proof}
It is shown in the proof of \cite[Theorem~5.6]{AFKKZ} that there exists a finite-dimensional $G$-module $V$ such that the $G$-action on $Y$ extends to a transitive action of the bigger group $G_V:=G \rightthreetimes V$ where the multiplication is given by
\[
(g_1,v_1)\cdot(g_2,v_2)=(g_1g_2, g_2^{-1}v_1+v_2).
\]
Let $v$ be a $B$-eigenvector in $V$ and let $H$ denote the corresponding $B$-root subgroup acting on $Y$.
We claim that at least one subgroup $H$ arising this way does not preserve the open $G$-orbit in $Y$.
Assume the converse.
Note that $V$ is the linear span of $G$-orbits of $B$-eigenvectors in $V$.
This implies that the subgroups conjugate to subgroups coming from $B$-eigenvectors by elements of $G$ generate the unipotent radical $V$ of~$G_V$.
So the group $G_V$ preserves the open $G$-orbit in $Y$, a contradiction.
\end{proof}

\subsection{Multiple \texorpdfstring{$B$}{B}-root subgroups of the same weight}

The following example shows that, unlike the toric case, a $B$-normalized LND $\partial$ on $\KK[X]$ is not uniquely determined by its weight (up to proportionality).

\begin{example}[$B$-root subgroups of the same weight] \label{ex_same_weights}
Take $X = \KK^3$ and let $G \subseteq \GL_3$ be the subgroup of all matrices of the form $\begin{pmatrix} sA & 0 \\ 0 & s^2 \end{pmatrix}$ where $s \in \KK^\times$ and $A \in \SL_2$.
It is easy to check (e.g., using Remark~\ref{rem_AHV_construction}) that $X$ is horospherical.
We have
\[
T = \lbrace \diag(ss_1,ss_1^{-1},s^2) \mid s,s_1 \in\nobreak \KK^\times \rbrace
\]
and let $\chi_1,\chi_2 \in \mathfrak X(T)$ be the characters defined by $\chi_1(t) = ss_1$, $\chi_2(t) = ss_1^{-1}$ for all $t = \diag(ss_1,ss_1^{-1},s^2) \in T$.
The algebra $\KK[X]^U$ is freely generated by the coordinate functions $x_2, x_3$ of weights $-\chi_2$, $-\chi_1-\chi_2$, respectively.
Consider the LNDs $\partial_1 = \partial/\partial x_1$ and $\partial_2 = x_2\partial/\partial x_3$ on $\KK[V]$.
Both $\partial_1$ and $\partial_2$ commute with $x_2\partial/\partial x_1$, which corresponds to~$U$, and hence are $B$-normalized.
It is easy to see that $\partial_1$ and $\partial_2$ are not proportional and have the same weight~$\chi_1$.
Note also that $\partial_1$ is vertical and $\partial_2$ is horizontal.
\end{example}

\begin{remark} \label{rem_same_weights}
In the above example any nontrivial linear combination $\partial = c_1\partial_1 + c_2\partial_2$ with $c_1,c_2 \in \KK$ is again a $B$-normalized LND of weight~$\chi_1$ (it is locally nilpotent because every $T$-semi\-invariant function in $\KK[X]$ has weight $k_1\chi_1 + k_2\chi_2$ with $k_1,k_2 \le 0$); hence we get a two-dimensional space of $B$-normalized LNDs on~$\KK[X]$ of the same weight.
Note that $\partial$ is horizontal whenever $c_2 \ne 0$.
\end{remark}

The next example shows that even a vertical $B$-root subgroup may be not uniquely determined by its weight.

\begin{example}[Vertical $B$-root subgroups of the same weight]
Take $X = \KK^4$ and let $G \subseteq \GL_4$ be the subgroup of all matrices of the form $\begin{pmatrix} A_1 & 0 \\ 0 & A_2 \end{pmatrix}$ where $A_1,A_2 \in \SL_2$.
Define characters $\chi_1,\chi_2 \in \mathfrak X(T)$ by $\chi_1(t) = s_1$, $\chi_2(t) = s_2$ for all $t = \diag(s_1,s_1^{-1},s_2,s_2^{-1}) \in T$.
The algebra $\KK[X]^U$ is freely generated by the coordinate functions $x_2, x_4$ of weights $\chi_1$, $\chi_2$, respectively.
Consider the LNDs $\partial_1 = x_4 \partial/\partial x_1$ and $\partial_2 = x_2\partial/\partial x_3$ on $\KK[X]$.
Both $\partial_1$ and~$\partial_2$ commute with $x_2\partial/\partial x_1$ and $x_4\partial/\partial x_3$, which correspond to the root subgroups in~$U$, and hence are $B$-normalized.
Moreover, $\partial_1$ and $\partial_2$ vanish on $\KK[X]^U$ and hence are vertical.
It is easy to see that $\partial_1$ and $\partial_2$ are not proportional and have the same weight~$\chi_1+\chi_2$.
We also note that any nontrivial linear combination of $\partial_1$ and $\partial_2$ is again a vertical $B$-normalized LND on~$\KK[X]$.
\end{example}

\begin{remark}
Another example of two vertical $B$-root subgroups of the same weight can be obtained in the situation of Remark~\ref{rem_X'} by taking appropriate replicas of $H$ and~$H_1$.
\end{remark}

\subsection{Weights of \texorpdfstring{$B$}{B}-root subgroups}

In this subsection we present straightforward necessary conditions for the set of weights of vertical $B$-root subgroups and that of horizontal ones.

\begin{proposition} \label{prop_weights_of_VRS}
Suppose that $\chi$ is the weight of a vertical $B$-root subgroup on~$X$.
Then each element of the set $\chi + \Gamma$ is again the weight of a vertical $B$-root subgroup on~$X$.
\end{proposition}

\begin{proof}
Let $\partial$ be a vertical $B$-normalized LND on~$\KK[X]$ of weight~$\chi$.
By definition, $\KK[X]^U \subseteq \Ker \partial$, hence it remains to apply Construction~\ref{constr_replicas}.
\end{proof}

\begin{remark}
When $\chi$ is the weight of a $T$-normalized central subgroup in~$U$ (see Construction~\ref{constr_central_subgroups}), the assertion of Proposition~\ref{prop_weights_of_VRS} was observed in~\cite[Proposition~8.1]{RvS}.
\end{remark}

\begin{proposition} \label{prop_weight_of_HRS}
Suppose that $H$ is a horizontal $B$-root subgroup on~$X$.
Then $\chi_H \in \mathfrak R(\mathcal E) \cap \Lambda^+$.
In particular, $\chi_H \in M$.
\end{proposition}

\begin{proof}
We have $\chi_H \in \Lambda^+$ by Proposition~\ref{prop_weight_is_dom}(\ref{prop_weight_is_dom_a}).
As $H$ is horizontal, its action induces a $T$-root subgroup on~$Z$, hence $\chi_H \in \mathfrak R(\mathcal E)$.
\end{proof}

The following example shows that the set of weights of horizontal $B$-root subgroups on~$X$ may be a proper subset of $\mathfrak R(\mathcal E) \cap \Lambda^+$.

\begin{example}[``Missing'' $B$-root subgroups] \label{ex_missing_root_subgroups}
Take $X = \KK^3$ and let $G \subseteq \GL_3$ be the subgroup of all matrices of the form $sA$ where $s \in \KK^\times$ and $A \in \SO_3$ with $\SO_3$ preserving the quadratic form $f = x_2^2 + 2x_1x_3$ on~$X$.
Then $\Lambda^+ = \ZZ_{\ge0}\alpha \oplus\nobreak \ZZ\chi$ where the characters $\chi,\alpha \in \mathfrak X(T)$ are defined by $\chi(t) = s$, $\alpha(t) = s_1$ for all $t = \diag(ss_1,s,ss_1^{-1}) \in T$; note that $\alpha$ is the unique root in~$\Pi$.
The algebra $\KK[X]^U$ is freely generated by the two functions $x_3$ and~$f$ of weights $\alpha - \chi$ and $-2\chi$, respectively, and so $\Gamma = \ZZ_{\ge0}\lbrace \alpha -\nobreak \chi, -2\chi \rbrace$.
It is not hard to see that $\mathfrak R(\mathcal E) \cap \Lambda^+ = 2\chi + \ZZ_{\ge0}(\alpha-\chi)$.
We now show that there is no nonzero $B$-normalized LND on $\KK[X]$ of weight~$2\chi$.
Indeed, by Proposition~\ref{prop_U^-} such an LND is uniquely determined by the image of~$x_1$.
Since $x_1$ is of weight $-\alpha - \chi$, its image should be a nonzero $T$-semiinvariant function of weight~$-\alpha + \chi$.
However there are no such functions because every $T$-semiinvariant function in $\KK[X]$ has weight of the form $k\alpha + l\chi$ with $l \le 0$.
Note also that for every $k \ge 0$ the formula $x_3^k\partial / \partial x_1$ defines a horizontal $B$-normalized LND of weight $\alpha + \chi + k(\alpha - \chi)$ on~$X$, hence the set of weights of horizontal $B$-root subgroups on~$X$ equals $(\mathfrak R(\mathcal E) \cap \Lambda^+) \setminus \lbrace 2\chi \rbrace$.
\end{example}

\subsection{\texorpdfstring{$B$}{B}-stable prime divisors moved by \texorpdfstring{$B$}{B}-root subgroups}

In this subsection we discuss various conditions under which a given $B$-stable prime divisor in~$X$ is moved or not moved by a $B$-root subgroup.

\begin{proposition} \label{prop_D_is_moved}
Suppose $H$ is a horizontal $B$-root subgroup on~$X$ and $\rho \in \mathcal E^1$ is such that $\chi_H \in \mathfrak R_\rho(\mathcal E)$.
Given $D \in \mathcal D^B$, the following conditions are equivalent.
\begin{enumerate}[label=\textup{(\arabic*)},ref=\textup{\arabic*}]
\item \label{prop_D_is_moved_1}
$\varkappa(D)$ is a positive multiple of~$\rho$.

\item \label{prop_D_is_moved_2}
$D$ is moved by~$H$.
\end{enumerate}
In particular, there is exactly one $D \in \mathcal D^B$ such that $\varkappa(D)$ is a positive multiple of~$\rho$.
\end{proposition}

\begin{proof}
(\ref{prop_D_is_moved_1})$\Rightarrow$(\ref{prop_D_is_moved_2})
Let $I \subseteq \KK[X]$ be the ideal of~$D$.
Then for every $\lambda \in \Gamma$ with $\langle \rho, \lambda \rangle > 0$ one has $f_\lambda \in I$ but $\partial^{\langle \rho, \lambda \rangle}f_\lambda \notin I$ where $\partial$ is the LND corresponding to~$H$.
Consequently, $I$ is $\partial$-unstable and hence $D$ is $H$-unstable.

(\ref{prop_D_is_moved_2})$\Rightarrow$(\ref{prop_D_is_moved_1})
In view of~(\ref{eqn_E(X)}) there is $D' \in \mathcal D^B$ such that $\varkappa(D')$ is a positive multiple of~$\rho$.
Then $D'$ is moved by~$H$ by the above argument.
Now Proposition~\ref{prop_open_B-orbit}(\ref{prop_open_B-orbit_b}) yields $D' = D$.

The last claim is also implied by Proposition~\ref{prop_open_B-orbit}(\ref{prop_open_B-orbit_b}).
\end{proof}

Combining Proposition~\ref{prop_D_is_moved} with~(\ref{eqn_Gamma(O)}) we obtain

\begin{corollary} \label{crl_hor_Gamma(O)}
Let $H$ be a horizontal $B$-root subgroup on~$X$ and let $D \in \mathcal D^B$ be moved by~$H$.
If $\chi_H \in \Gamma(O)$ then $D \in \mathcal D^G$.
\end{corollary}

The next proposition shows that some colors of~$X$ cannot be moved by $B$-root subgroups.

\begin{proposition} \label{prop_frozen_colors}
Suppose that $D \in \mathcal D^B$ satisfies one of the following conditions:
\begin{enumerate}[label=\textup{(\alph*)},ref=\textup{\alph*}]
\item \label{prop_frozen_colors_a}
$\varkappa(D)$ is not proportional to any $\rho \in \mathcal E^1$;
\item \label{prop_frozen_colors_b}
$\varkappa(D) = c\varkappa(D')$ for some $D' \in \mathcal D^B \setminus \lbrace D \rbrace$ and $c > 0$;
\item \label{prop_frozen_colors_c}
$\langle \varkappa(D), \lambda \rangle \ge 0 $ for all $\lambda \in M \cap \Lambda^+$.
\end{enumerate}
Then there is no $B$-root subgroup on~$X$ that moves~$D$.
Moreover, $D \in \mathcal D$, i.e. $D$ is necessarily a color of~$X$.
\end{proposition}

\begin{proof}
(\ref{prop_frozen_colors_a},\,\ref{prop_frozen_colors_b})
In both cases, the first claim follows directly from Proposition~\ref{prop_D_is_moved} and the second one is implied by Proposition~\ref{prop_rho_G-stable}.

(\ref{prop_frozen_colors_c})
In view of part~(\ref{prop_frozen_colors_a}) it remains to consider the case where $\varkappa (D)$ is a positive multiple of some~$\rho \in \mathcal E^1$.
Assume that a $B$-root subgroup $H$ on $X$ moves~$D$.
Then $\chi_H \in \mathfrak R_\rho(\mathcal E)$ by Proposition~\ref{prop_D_is_moved}, and so $\langle \varkappa(D),\chi_H \rangle < 0$.
On the other hand, $\chi_H \in M \cap \Lambda^+$ by Remark~\ref{rem_DR_are_in_M} and Proposition~\ref{prop_weight_is_dom}(\ref{prop_weight_is_dom_a}), therefore $\langle \varkappa(D), \chi_H \rangle \ge 0$, a contradiction.
Next, since $\Gamma(O) \subseteq M \cap \Lambda^+$, it follows from~(\ref{eqn_Gamma(O)}) that $\varkappa(D) \in \QQ_{\ge0}\lbrace \varkappa(D') \mid D' \in \mathcal D \rbrace$.
Thanks to Proposition~\ref{prop_rho_G-stable}, the latter is possible only if $D \in \mathcal D$.
\end{proof}

Recall colors of type~$a$ on~$X$; see Definition~\ref{def_type_a}.

\begin{corollary} \label{crl_frozen_colors}
If a color $D \in \mathcal D$ is not of type~$a$ then there is no $B$-root subgroup on~$X$ that moves~$D$.
\end{corollary}

\begin{proof}
This follows from Propositions~\ref{prop_types_of_colors}(\ref{prop_types_of_colors_c}) and~\ref{prop_frozen_colors}(\ref{prop_frozen_colors_c}).
\end{proof}

The next corollary indicates a large class of affine spherical $G$-varieties containing no colors that can be moved by a $B$-root subgroup.
By Proposition~\ref{prop_Gamma(O)} and Remarks~\ref{rem_colors_of_type_a},\,\ref{rem_Sigma_is_empty}, this class includes all affine spherical $G$-varieties having no colors of type~$a$ and in particular all affine horospherical $G$-varieties.
Moreover, $X$ belongs to this class when $O$ is a symmetric space; see~\cite[Definition~26.1 and Proposition~26.24]{Tim}.

\begin{corollary} \label{crl_no_moved_colors}
Suppose that $\Gamma(O) = M \cap \Lambda^+$.
Then no color in~$X$ can be moved by a $B$-root subgroup.
\end{corollary}

\begin{proof}
The claim is implied by Proposition~\ref{prop_frozen_colors}(\ref{prop_frozen_colors_c}) along with~(\ref{eqn_Gamma(O)}).
\end{proof}

\begin{remark}
It follows from~(\ref{eqn_Gamma(O)}) that $\mathcal D \ne \varnothing$ whenever $M \cap \Lambda^+ \ne M$.
\end{remark}

In view of Corollary~\ref{crl_no_moved_colors} it is natural to ask whether it happens at all that a color on an affine spherical $G$-variety is moved by a $B$-root subgroup.
The next example shows that such situations do occur.

\begin{example}[Colors moved by $B$-root subgroups] \label{ex_moved_colors}
Let $G_1 = G_2 = \GL_2$, $X = \mathrm{Mat}_{2\times 2}(\KK)$ and consider the action of $G_1 \times G_2$ on $X$ given by $((g_1,g_2),x) \mapsto g_1xg_2^{-1}$.
Restrict this action to the subgroup $G = F_1 \times F_2$ where $F_1 = \SL_2 \subseteq G_1$ and $F_2 \simeq (\KK^\times)^2$ is the diagonal torus in~$G_2$.
Choose the Borel subgroup $B = B_1 \times F_2 \subseteq G$ where $B_1$ consists of all upper-triangular matrices in~$F_1$.
Then $X$ is an affine spherical $G$-variety and $\mathcal D = \lbrace D_1,D_2 \rbrace$ with $D_1 = \lbrace x_{21} = 0 \rbrace$, $D_2 = \lbrace x_{22} = 0 \rbrace$ where $x_{ij}$ stands for the $(ij)$th matrix element of~$X$.
Note also that $\mathcal D^B = \lbrace D_1, D_2, D_3 \rbrace$ with $D_3 = \lbrace \det = 0 \rbrace$.
Now the action of the subgroup of all lower (resp. upper) unitriangular matrices in~$G_2$ induces a $B$-root subgroup $H_1$ (resp.~$H_2$) on~$X$ that moves~$D_1$ (resp.~$D_2$) and preserves the open $G$-orbit $O = \GL_2 \subseteq X$.
Of course, both colors $D_1,D_2$ are of type~$a$; moreover, $\mathcal D = \mathcal D_\alpha$ where $\alpha$ is the unique simple root of~$G$.
Observe that both $H_1,H_2$ are normalized by the whole~$G$, so in fact they are $G$-root subgroups on~$X$.
Besides, it is worth mentioning that the $B$-normalized LNDs $\partial/\partial x_{11}$ and $\partial/\partial x_{12}$ (see Construction~\ref{constr_SM}) do not vanish on~$\det$, therefore the respective $B$-root subgroups $H_3,H_4$ on~$X$ move~$D_3$.

Let $T_1$ be the subgroup of diagonal matrices in $F_1$, so that $T = T_1 \times F_2$ is a maximal torus in~$G$.
Let $\omega = \alpha/2 \in \mathfrak X(T_1)$ be the fundamental weight of~$F_1$, so that $\omega(t_1) = s$ for all $t_1 = \diag(s,s^{-1}) \in T_1$.
Choose also the basis $\chi_1,\chi_2 \in \mathfrak X(F_2)$ such that $\chi_i(t_2) = s_i$ for $i=1,2$ and all $t_2 = \diag(s_1,s_2) \in F_2$.
The algebra $\KK[X]^U$ is freely generated by the functions $x_{21},x_{22},\det$ of weights $\omega+\chi_1$, $\omega+\chi_2$, $\chi_1+\chi_2$, respectively, therefore
\[
\Gamma = \ZZ_{\ge0} \lbrace \omega + \chi_1, \omega + \chi_2, \chi_1 + \chi_2 \rbrace.
\]
We note also that $\Gamma(O) = \ZZ_{\ge0} \lbrace \omega + \chi_1, \omega + \chi_2 \rbrace \oplus \ZZ(\chi_1+\chi_2)$.
The weights of $H_1,H_2,H_3,H_4$ are $\chi_2 - \chi_1$, $\chi_1 - \chi_2$, $\omega - \chi_1$, $\omega - \chi_2$, respectively.
\end{example}

Despite the above-discussed negative results on moving colors by $B$-root subgroups, we propose the following

\begin{conjecture} \label{conj_G-stable}
For every $D \in \mathcal D^G$ there is a $B$-root subgroup on~$X$ that moves~$D$.
\end{conjecture}

In Theorem~\ref{thm_moved_divisors} below we prove this conjecture when $X$ is horospherical.

\subsection{The case of \texorpdfstring{$G$}{G}-saturated \texorpdfstring{$\Gamma$}{Gamma}}

The monoid $\Gamma$ is said to be \textit{$G$-saturated} if $\Gamma = \ZZ\Gamma \cap \Lambda^+$.
In other words, $\Gamma$ equals the intersection of $\Lambda^+$ with a sublattice of~$\mathfrak X(T)$.

\begin{remark} \label{rem_G-saturated}
$\Gamma$ is $G$-saturated if and only if $\mathcal E = \QQ_{\ge0} \lbrace \iota(\alpha^\vee) \mid \alpha \in \Pi \rbrace$.
\end{remark}

\begin{proposition} \label{prop_G-saturated}
Suppose that $\Gamma$ is $G$-saturated.
Then every $B$-root subgroup on~$X$ is vertical.
\end{proposition}

\begin{proof}
Recall from Remark~\ref{rem_DR_are_in_M} that $\mathfrak R(\mathcal E) \subseteq \ZZ\Gamma$.
Then
\[
\mathfrak R(\mathcal E) \cap \Lambda^+ = \mathfrak R(\mathcal E) \cap \ZZ\Gamma \cap \Lambda^+ = \mathfrak R(\mathcal E) \cap \Gamma = \varnothing,
\]
where the latter equality holds because $\mathfrak R(\mathcal E) \cap \mathcal G = \varnothing$.
Now the claim follows from Proposition~\ref{prop_weight_of_HRS}.
\end{proof}

\begin{remark}
If $\Gamma$ is $G$-saturated then, combining Remark~\ref{rem_G-saturated} with Propositions~\ref{prop_rho_G-stable} and~\ref{prop_frozen_colors}(\ref{prop_frozen_colors_c}), we find that $\mathcal D^G = \varnothing$.
So Proposition~\ref{prop_G-saturated} does not contradict Conjecture~\ref{conj_G-stable}.
\end{remark}

\begin{corollary} \label{crl_rk=1}
Suppose that $\rk \Gamma =1$ and $(G,G)$ is nontrivial and acts on~$X$ with a finite kernel.
Then every $B$-root subgroup on $X$ is vertical.
\end{corollary}

\begin{proof}
The hypotheses imply that $\Gamma = \ZZ_{\ge 0} \lambda$ for a dominant weight $\lambda \in \Lambda^+$ that restricts nontrivially to~$(G,G)$.
Then $-\lambda \notin \Lambda^+$ and hence $\Gamma = \ZZ\Gamma \cap \Lambda^+$, which yields the claim thanks to Proposition~\ref{prop_G-saturated}.
\end{proof}

\begin{corollary}
Suppose that $G=\SL_2$ and $G$ acts nontrivially on~$X$.
Then every $B$-root subgroup on $X$ is equivalent to~$U$.
\end{corollary}

\begin{proof}
Clearly, the conditions of Corollary~\ref{crl_rk=1} are satisfied, therefore every $B$-root subgroup on $X$ is vertical and hence preserves generic $U$-orbits in~$X$.
But the group $U$ is one-dimensional, so every $B$-root subgroup on~$X$ is equivalent to~$U$ in this case.
\end{proof}

\begin{remark}
Along with the case $G = \SL_2$, Corollary~\ref{crl_rk=1} may be applied to other natural classes of spherical varieties.
Recall that an $HV$-variety is the closure of the orbit of a highest weight vector in a simple module $V$ of a semisimple group~$G$.
It is known that every $HV$-variety is a normal cone consisting of two orbits---the open orbit and the origin---and the weight monoid of such a variety is generated by the dual of the highest weight of~$V$; see~\cite[\S\,1]{VP72} for details.
By Corollary~\ref{crl_rk=1}, there are only vertical $B$-root subgroups in this case.
An important example of $HV$-varieties is given by Grassmann cones, which are cones of highest weight vectors in the fundamental representations of the group~$\SL_n$.
Another example is the nondegenerate quadratic cone with the action of the group~$\SO_n$.
\end{remark}

\subsection{\texorpdfstring{$B$}{B}-root subgroups arising in the study of spherical subgroups}
\label{subsec_spherical_subgroups}

This subsection serves as a complement to \S\,\ref{subsec_simplest_constructions}.
Here we present one more setting that yields nontrivial examples of affine spherical $G$-varieties equipped with $B$-root subgroups on them.
Besides, we state an open problem in this setting, which provides an extra motivation for studying $B$-root subgroups.

Let $P \subseteq G$ be a parabolic subgroup such that $P \supseteq B^-$ and let $P_u$ denote the unipotent radical of~$P$.
Let $L$ be the unique Levi subgroup of~$P$ containing~$T$ and let $\Pi_L \subseteq \Pi$ be the set of simple roots of~$L$.
Let $Q \subseteq P$ be a closed subgroup with unipotent radical $Q_u$ and a Levi subgroup~$K$.
Suppose that $K \subseteq L$ and $Q_u \subseteq P_u$.

It is known from \cite[Proposition~I.1]{Br87} and~\cite[Theorem~1.2]{Pa94} that the following conditions are equivalent:
\begin{enumerate}
\item
$Q$ is a spherical subgroup of~$G$, i.e. $G/Q$ is a spherical $G$-variety;

\item
$P/Q$ is a spherical $L$-variety (which is automatically smooth and affine);

\item
$P_u/Q_u$ is a spherical $S$-variety (which is an affine space) for a certain reductive subgroup $S \subseteq K$ uniquely determined up to conjugacy by the pair $(L,K)$.
\end{enumerate}

Put $B_L = B \cap L$ and $B_S = B \cap S$, then $B_L$ is a Borel subgroup of~$L$ and the connected component of the identity $B_S^0 \subseteq B_S$ is a Borel subgroup of~$S$ under an appropriate choice of~$S$ within its conjugacy class in~$K$.

Now for every $\alpha \in \Pi \setminus \Pi_L$ the group $U_{-\alpha} \subseteq P$ naturally acts on $P/Q$ and is normalized by~$B_L$.
When this action is nontrivial, it provides a $B_L$-root subgroup on~$P/Q$.
Similarly, $U_{-\alpha} \subseteq P_u$ naturally acts on $P_u/Q_u$ and is normalized by~$B_S^0$.
Again, if this action is nontrivial then it provides a $B_S^0$-root subgroup on $P_u/Q_u$.
We note that in the case $P = B^-$ both groups $L$ and $S^0$ are tori and the above-mentioned root subgroups were considered in~\cite{GP}.

Thanks to~\cite[Lemma~1.4]{Mon}, there is a $K$-equivariant (and hence $S$-equivariant) isomorphism $P_u/Q_u \simeq \mathfrak p_u/\mathfrak q_u$ where $\mathfrak p_u = \Lie P_u$ and $\mathfrak q_u = \Lie Q_u$.
Thus we obtain a $B_S^0$-normalized action of~$U_{-\alpha}$ on the spherical $S$-module $\mathfrak p_u / \mathfrak q_u$.
A criterion for this action to be nontrivial is given by~\cite[Lemma~6.6]{Avd}.
When $U_{-\alpha}$ acts nontrivially on~$\mathfrak p_u / \mathfrak q_u$, it follows from~\cite[Lemma~6.7]{Avd} that the induced $B^0_S$-root subgroup is horizontal if and only if $\alpha$ is a spherical root of~$G/Q$.
As discussed in~\cite{Avd}, the set $\Pi \cap \Sigma(G/Q)$ of simple spherical roots of $G/Q$ plays an essential role in computing several key combinatorial invariants of~$G/Q$.
Thus, it is important to find methods for computing $\Pi \cap \Sigma(G/Q)$ itself, which is equivalent to the following open problem in our setting.

\begin{problem}
Under the above notation and assumptions suppose that $U_{-\alpha}$ acts nontrivially on~$P_u / Q_u$.
Determine whether the induced $B_S^0$-root subgroup on~$P_u / Q_u$ is vertical or horizontal.
\end{problem}


\section{Standard \texorpdfstring{$B$}{B}-root subgroups}
\label{sec5}

In this section, we present and discuss a general construction of horizontal $B$-root subgroups on affine spherical $G$-varieties, which we call standard, and provide several applications.

Let $X$ be an affine spherical $G$-variety and retain all the notation from~\S\,\ref{subsec_ASV_notation}.
Suppose $\partial$ is a horizontal $B$-normalized LND on~$\KK[X]$ of weight~$\mu$ and $\rho \in \mathcal E^1$ is such that $\mu \in \mathfrak R_\rho(\mathcal E)$.
Recall that the restriction of~$\partial$ to $\KK[X]^U$ is $T$-normalized.
Thanks to Theorem~\ref{thm_T-root_subgroups}(\ref{thm_T-root_subgroups_a}) and formula~(\ref{eqn_T-norm_LND}), up to rescaling, we have
\begin{equation} \label{eqn_der_on_f_lambda}
\partial(f_\lambda)=\langle \rho, \lambda \rangle f_\mu f_\lambda
\end{equation}
for all $\lambda \in \Gamma$.
Our construction naturally generalizes formula~(\ref{eqn_der_on_f_lambda}) to the whole $\KK[X]$ when $\mu \in \Gamma(O)$.

\subsection{Description of the construction}

Take any $\mu \in \mathfrak R(\mathcal E) \cap \Gamma(O)$ and let $\rho \in \mathcal E^1$ be such that $\mu \in \mathfrak R_\rho(\mathcal E)$.
We define a linear map $\partial_\mu \colon \KK[X] \to \KK[O]$ as follows.
For every $\lambda \in \Gamma$ and $g \in \KK[X]_\lambda$, we put
\begin{equation} \label{eqn_horo_LND}
\partial_\mu(g) = \langle \rho, \lambda \rangle f_\mu g.
\end{equation}
As $f_\mu$ is a highest-weight vector, $\partial_\mu$ commutes with~$\mathfrak u$ and hence is $B$-normalized.

\begin{proposition}
The image of $\partial_\mu$ is contained in~$\KK[X]$, so that $\partial_\mu$ is a well-defined linear map of $\KK[X]$ to itself.
\end{proposition}

\begin{proof}
Recall that $\langle \rho, \mu \rangle = -1$ and $\langle \rho', \mu \rangle \ge 0$ for all $\rho' \in \mathcal E^1 \setminus \lbrace \rho \rbrace$.
If $\langle \rho, \lambda \rangle = 0$ then $\partial_\mu(g) =\nobreak 0$ for all $g \in \KK[X]_\lambda$, so we assume in what follows that $\langle \rho, \lambda \rangle \ge 1$.
Take $\nu \in\nobreak \Gamma(O)$ such that $\KK[O]_\nu$ is contained in the linear span of $\KK[O]_\lambda \cdot \KK[O]_\mu$ and consider the corresponding tail $\tau = \lambda + \mu - \nu \in \mathcal T$.
Let $\rho' \in \mathcal E^1$ be an arbitrary element; we need to show that
\begin{equation} \label{eqn_ge0}
\langle \rho', \nu \rangle \ge 0.
\end{equation}
We know from~(\ref{eqn_E(X)}) that there is $D \in \mathcal D^B$ such that $\varkappa(D)$ is a positive multiple of~$\rho'$.
If $D \in\nobreak \mathcal D$ then (\ref{eqn_ge0}) holds by~(\ref{eqn_Gamma(O)}), so in what follows we assume $D \in \mathcal D^G$.
Then Proposition~\ref{prop_tail_cone} and Corollary~\ref{crl_tail_cone}(\ref{crl_tail_cone_b}) yield $\langle \rho', \tau \rangle \le 0$, therefore
\[
\langle \rho', \nu \rangle = \langle \rho', \lambda + \mu - \tau \rangle \ge \langle \rho', \lambda\rangle + \langle \rho', \mu \rangle,
\]
and it remains to prove that the latter expression is nonnegative.
If $\rho' \ne \rho$ then both summands are nonnegative.
If $\rho' = \rho$ then $\langle \rho', \lambda \rangle \ge 1$ and $\langle \rho',\mu \rangle = -1$, and we are done.
\end{proof}

\begin{proposition} \label{prop_partial_mu}
The map $\partial_\mu$ is a derivation of~$\KK[X]$ if and only if $\rho \in \mathcal T^\perp$.
Moreover, under these conditions $\partial_\mu$ is locally nilpotent.
\end{proposition}

\begin{proof}
Take arbitrary $\lambda, \lambda' \in \Gamma$ and $g \in \KK[X]_\lambda$, $g' \in \KK[X]_{\lambda'}$.
Then $gg' = h+\sum \limits_{i=1}^k h_i$ where $h \in \KK[X]_{\lambda + \lambda'}$ and $h_i \in \KK[X]_{\nu_i} \setminus \lbrace 0 \rbrace$ for some pairwise distinct weights $\nu_1,\ldots,\nu_k \in \Gamma \setminus \lbrace \lambda + \lambda' \rbrace$.
Observe that $\lambda+\lambda' - \nu_i \in \mathcal T$ for all $i = 1,\ldots, k$.
We have
\begin{multline} \label{eqn_derivation}
\partial_\mu(gg') - g\partial_\mu(g') - g'\partial_\mu(g) =
\langle \rho, \lambda + \lambda' \rangle f_\mu h + \sum \limits_{i=1}^k \langle \rho, \nu_i \rangle f_\mu h_i - \langle \rho, \lambda' \rangle f_\mu gg' - \langle \rho, \lambda \rangle f_\mu gg' = \\
\langle \rho, \lambda + \lambda' \rangle f_\mu h + \sum \limits_{i=1}^k \langle \rho, \nu_i \rangle f_\mu h_i - \langle \rho, \lambda + \lambda' \rangle f_\mu (h +\sum \limits_{i=1}^k h_i) = -\sum \limits_{i=1}^k \langle \rho, \lambda + \lambda' - \nu_i \rangle f_\mu h_i.
\end{multline}
If $\rho \in \mathcal T^\perp$ then the last expression vanishes and hence $\partial_\mu$ is a derivation.

Conversely, suppose that $\partial_\mu$ is a derivation of~$\KK[X]$.
Then the last expression in~(\ref{eqn_derivation}) vanishes; dividing it by $f_\mu$ we obtain $\sum \limits_{i=1}^k \langle \rho, \lambda + \lambda' - \nu_i \rangle h_i = 0$, which implies $\langle \rho, \lambda + \lambda' - \nu_i \rangle = 0$ for all $i = 1,\ldots,k$.
Since every tail $\tau \in \mathcal T$ can be realized as $\lambda + \lambda' - \nu_i$ for an appropriate choice of $\lambda,\lambda',g,g',i$, we obtain $\langle \rho, \tau \rangle = 0$.

Finally, consider the decomposition $\KK[X] = \bigoplus \limits_{i=0}^\infty \KK[X]^{(i)}$ where $\KK[X]^{(i)} = \bigoplus \limits_{\lambda \in \Gamma : \langle \rho, \lambda \rangle = i} \KK[X]_\lambda$.
Then it is easy to see that $\partial_\mu(\KK[X]^{(0)}) = 0$ and $\partial_\mu(\KK[X]^{(i+1)}) \subseteq \KK[X]^{(i)}$ for all $i \in \ZZ_{\ge0}$, so $\partial_\mu$ is locally nilpotent.
\end{proof}

\begin{corollary} \label{cor_partial_mu}
Under the conditions of Proposition~\textup{\ref{prop_partial_mu}}, $\partial_\mu$ is a horizontal $B$-normalized LND of~$\KK[X]$ and the corresponding $B$-root subgroup moves a $G$-stable prime divisor on~$X$.
\end{corollary}

\begin{proof}
Formula~(\ref{eqn_horo_LND}) implies that $\partial_\mu$ does not vanish on~$\KK[X]^U$, hence $\partial_\mu$ is horizontal.
The last assertion follows from Corollary~\ref{crl_hor_Gamma(O)}.
\end{proof}

\begin{definition}
A nonzero $B$-normalized LND on~$\KK[X]$ is called \textit{standard} if it is proportional to the LND $\partial_\mu$ given by~(\ref{eqn_horo_LND}) for some $\rho \in \mathcal E^1 \cap \mathcal T^\perp$ and $\mu \in \mathfrak R_\rho(\mathcal E) \cap \Gamma(O)$.
A~$B$-root subgroup on~$X$ is called \textit{standard} if it corresponds to a standard $B$-normalized LND on~$\KK[X]$.
\end{definition}

\begin{remark} \label{rem_nonstandard}
In general, not all horizontal $B$-root subgroups on affine spherical $G$-varieties that move a $G$-stable prime divisor are standard.
Moreover, this remains valid even for affine horospherical $G$-varieties.
For instance, in the situation of Example~\ref{ex_same_weights} and Remark~\ref{rem_same_weights}, where $X$ is horospherical, the LND $\partial_2 + c\partial_1$ is standard if and only if $c=0$.
In Example~\ref{ex_missing_root_subgroups}, the LND $x_3^k\partial/\partial x_1$ is not standard for all $k \ge 0$.
\end{remark}

\subsection{\texorpdfstring{$G$}{G}-stable prime divisors moved by standard \texorpdfstring{$B$}{B}-root subgroups}

The next proposition provides a necessary and sufficient combinatorial condition for a $G$-stable prime divisor on~$X$ to be moved by a standard $B$-root subgroup.

\begin{proposition} \label{prop_moving_D}
Given $D \in \mathcal D^G$, the following conditions are equivalent.
\begin{enumerate}[label=\textup{(\arabic*)},ref=\textup{\arabic*}]
\item \label{prop_moving_D_1}
There is a standard $B$-root subgroup on~$X$ that moves~$D$.

\item \label{prop_moving_D_2}
$\varkappa(D) \in \mathcal T^\perp$.
\end{enumerate}
\end{proposition}

\begin{proof}
(\ref{prop_moving_D_1})$\Rightarrow$(\ref{prop_moving_D_2})
Suppose $D$ is moved by a standard $B$-root subgroup of weight~$\mu$ and let $\rho \in \mathcal E^1$ be such that $\mu \in \mathfrak R_\rho(\mathcal E)$.
Then $\varkappa(D)$ is a positive multiple of~$\rho$ by Proposition~\ref{prop_D_is_moved}, which implies $\varkappa(D) \in \mathcal T^\perp$ by Proposition~\ref{prop_partial_mu}.

(\ref{prop_moving_D_2})$\Rightarrow$(\ref{prop_moving_D_1})
Let $\widetilde{\mathcal E}$ be the cone in $N_\QQ$ dual to $\QQ_{\ge 0}\Gamma(O)$; observe that $\widetilde{\mathcal E} \subseteq \mathcal E$.
It follows from~(\ref{eqn_Gamma(O)}) that $\widetilde{\mathcal E}$ is generated by the set $\lbrace \varkappa(D') \mid D' \in \mathcal D \rbrace$.
Put $\rho = \varkappa(D)$; then Proposition~\ref{prop_rho_G-stable} yields $\rho \in \mathcal E^1 \setminus \widetilde{\mathcal E}$.
Note that Lemma~\ref{lemma_two_cones} is applicable in this situation, hence the set $\mathfrak R_\rho(\mathcal E) \cap \Gamma(O)$ contains infinitely many elements.
We choose such an element $\mu$ and consider the map $\partial_\mu$ given by~(\ref{eqn_horo_LND}).
As $\rho \in \mathcal T^\perp$, Proposition~\ref{prop_partial_mu} implies that $\partial_\mu$ is a standard $B$-normalized LND on $\KK[X]$.
Thanks to Proposition~\ref{prop_D_is_moved}, the corresponding $B$-root subgroup on~$X$ moves~$D$.
\end{proof}

\subsection{\texorpdfstring{$G$}{G}-root subgroups}

In this subsection we apply the construction of standard $B$-root subgroups to obtain a partial description of $G$-root subgroups on~$X$.
Recall from~\S\,\ref{subsec_root_subgroups} that the weight of every $G$-root subgroup on~$X$ belongs to~$\mathfrak X(G)$, which is identified with a subgroup of~$\mathfrak X(T)$.

\begin{proposition}[{compare with~\cite[Lemma~5.1]{LP}}]
\label{prop_G-RS_is_hor}
Suppose $H$ is a $G$-root subgroup on $X$.
Then $H$ is horizontal as a $B$-root subgroup and uniquely determined by its weight among the $B$-root subgroups on~$X$.
Moreover, $\chi_H \in \mathfrak R(\mathcal E) \cap \mathfrak X(G)$ and in particular $\chi_H \in M$.
\end{proposition}

\begin{proof}
Clearly, $H$ is $B^-$-normalized, therefore by Proposition~\ref{prop_U^-} the corresponding $B^-$-normalized LND $\partial$ on~$\KK[X]$ is uniquely determined by its restriction to~$\KK[X]^U$, which in turn is $T$-normalized.
By Theorem~\ref{thm_T-root_subgroups}(\ref{thm_T-root_subgroups_a}), up to proportionality, the latter restriction is uniquely determined by its weight.
Note that $\partial$ acts nontrivially on $\KK[X]^U$, so $H$ is horizontal.
By Proposition~\ref{prop_weight_is_dom}(\ref{prop_weight_is_dom_b}), there are no other $B$-root subgroups on~$X$ of weight~$\chi_H$.
The last claim is implied by Proposition~\ref{prop_weight_of_HRS}.
\end{proof}

\begin{corollary} \label{crl_no_G-root_subgroups}
If $G$ is semisimple then there are no $G$-root subgroups on~$X$.
\end{corollary}

\begin{proof}
The claim follows from $\mathfrak X(G) = \lbrace 0 \rbrace$ and $0 \notin \mathfrak R(\mathcal E)$.
\end{proof}

The next result shows that all $G$-normalized LNDs on $\KK[X]$ act in a rather simple way.

\begin{proposition} \label{prop_G-norm_LND}
Let $\partial$ be a $G$-normalized LND on~$\KK[X]$ of weight~$\mu$ and let $\rho \in \mathcal E^1$ be such that $\mu \in \mathfrak R_\rho(\mathcal E)$.
Then there is $c \in \KK^\times$ such that, for every $\lambda \in \Gamma$, $\partial$ acts on $\KK[X]_\lambda$ as follows:
\begin{itemize}
\item
if $\langle \rho, \lambda \rangle = 0$ then $\partial(\KK[X]_\lambda) = 0$;

\item
if $\langle \rho, \lambda \rangle > 0$ then the restriction of $\partial$ to $\KK[X]_\lambda$ is a $(G,G)$-equivariant isomorphism $\KK[X]_\lambda \xrightarrow{\sim} \KK[X]_{\lambda + \mu}$ such that $\partial(f_\lambda) = c\langle \rho, \lambda \rangle f_{\lambda+\mu}$.
\end{itemize}
\end{proposition}

\begin{proof}
The assertion is implied by formula~(\ref{eqn_der_on_f_lambda}) for the restriction of $\partial$ to $\KK[X]^U$ and the $(G,G)$-invariance of~$\partial$.
\end{proof}

Let $\widetilde M \subseteq \mathfrak X(G)$ be the lattice of weights of $G$-semiinvariant functions in~$\KK(X)$.
Since every such function is automatically regular on~$O$, we have
\begin{equation} \label{eqn_G-semiinv}
\widetilde M = \mathfrak X(G) \cap \Gamma(O).
\end{equation}
The next result provides a description of all weights of $G$-root subgroups on~$X$ that belong to~$\widetilde M$.

\begin{theorem} \label{thm_G-root_subgroups}
For a weight $\mu \in \widetilde M$, the following conditions are equivalent.
\begin{enumerate}[label=\textup{(\arabic*)},ref=\textup{\arabic*}]
\item \label{thm_G-root_subgroups_1}
$\mu$ is the weight of a $G$-root subgroup on~$X$.

\item \label{thm_G-root_subgroups_2}
$\mu \in \mathfrak R_\rho(\mathcal E)$ for some $\rho \in \mathcal E^1 \cap \mathcal T^\perp$.
\end{enumerate}
Moreover, under these conditions the $G$-root subgroup of weight~$\mu$ on~$X$ is automatically standard as a $B$-root subgroup.
\end{theorem}

\begin{proof}
(\ref{thm_G-root_subgroups_1})$\Rightarrow$(\ref{thm_G-root_subgroups_2})
Let $H$ be a $G$-root subgroup on~$X$ of weight $\mu \in \widetilde M$ and let $\partial$ be the corresponding $G$-normalized LND on $\KK[X]$.
We know from Proposition~\ref{prop_G-RS_is_hor} that $H$ is horizontal as a $B$-root subgroup and there is $\rho \in \mathcal E^1$ such that $\mu \in \mathfrak R_\rho(\mathcal E)$.
In view of~(\ref{eqn_G-semiinv}), the function $f_\mu $ belongs to~$\KK[O]$ and is $G$-semiinvariant, therefore for every $\lambda \in \Gamma$ with $\langle \rho, \lambda \rangle > 0$ the map $g \mapsto f_\mu g$ induces a $(G,G)$-equivariant isomorphism $\KK[X]_\lambda \to \KK[X]_{\lambda+\mu}$.
Now Proposition~\ref{prop_G-norm_LND} yields $\partial = c\partial_\mu$ for some $c \in \KK^\times$, hence $\rho \in \mathcal T^\perp$ by Proposition~\ref{prop_partial_mu}.

(\ref{thm_G-root_subgroups_2})$\Rightarrow$(\ref{thm_G-root_subgroups_1})
By~(\ref{eqn_G-semiinv}) and Proposition~\ref{prop_partial_mu}, $\partial_\mu$ is a $B$-normalized LND on $\KK[X]$ of weight~$\mu$.
As $f_\mu$ is $G$-semiinvariant, $\partial_\mu$ is in fact $G$-normalized.
\end{proof}

\begin{remark}
Describing all weights of $G$-root subgroups on~$X$ that do not belong to $\widetilde M$ remains an open problem.
Such $G$-root subgroups may exist: in Example~\ref{ex_moved_colors} we have seen $G$-root subgroups on~$X$ that move a color; their weights do not belong to~$\Gamma(O)$ (and hence to $\widetilde M$) by Corollary~\ref{crl_hor_Gamma(O)}.
\end{remark}

Combining Theorem~\ref{thm_G-root_subgroups} with~(\ref{eqn_G-semiinv}) we obtain the following complete description of weights of \textit{all} $G$-root subgroups for a large class of affine spherical $G$-varieties.

\begin{corollary} \label{crl_G-root_subgroups}
Suppose that $\Gamma(O) = M \cap \Lambda^+$.
Then, for a weight $\mu \in \mathfrak X(G)$, the following conditions are equivalent.
\begin{enumerate}[label=\textup{(\arabic*)},ref=\textup{\arabic*}]
\item \label{crl_G-root_subgroups_1}
$\mu$ is the weight of a $G$-root subgroup on~$X$.

\item \label{crl_G-root_subgroups_2}
$\mu \in \mathfrak R_\rho(\mathcal E)$ for some $\rho \in \mathcal E^1 \cap \mathcal T^\perp$.
\end{enumerate}
Moreover, all $G$-root subgroups on~$X$ are standard as $B$-root subgroups.
\end{corollary}

A refined version of Corollary~\ref{crl_G-root_subgroups} in the horospherical case is provided by Proposition~\ref{prop_horo_weights}(\ref{prop_horo_weights_a}) below.

\subsection{The horospherical case}
\label{subsec_horo_case}

Let $X$ be an affine horospherical $G$-variety (see~\S\,\ref{subsec_aff_horo}) and recall from Theorem~\ref{thm_horospherical}, Remark~\ref{rem_Sigma_is_empty}, and Proposition~\ref{prop_Gamma(O)} that $\mathcal T = \lbrace 0 \rbrace$ and $\Gamma(O) = M \cap \Lambda^+$.
In particular, we automatically get $\mathcal E^1 \subseteq \mathcal T^\perp$.

\begin{proposition} \label{prop_horo_weights}
The following assertions hold.
\begin{enumerate}[label=\textup{(\alph*)},ref=\textup{\alph*}]
\item \label{prop_horo_weights_a}
The set of weights of $G$-root subgroups on~$X$ is $\mathfrak R(\mathcal E) \cap \mathfrak X(G)$.
Moreover, all $G$-root subgroups on~$X$ are standard as $B$-root subgroups.

\item \label{prop_horo_weights_b}
The set of weights of horizontal $B$-root subgroups on~$X$ is $\mathfrak R(\mathcal E) \cap \Lambda^+$.
Moreover, for every $\mu \in \mathfrak R(\mathcal E) \cap \Lambda^+$ there is a standard $B$-root subgroup on~$X$ of weight~$\mu$.
\end{enumerate}
\end{proposition}

\begin{proof}
(\ref{prop_horo_weights_a})
The claim follows directly from Corollary~\ref{crl_G-root_subgroups}.

(\ref{prop_horo_weights_b})
Recall from Remark~\ref{rem_DR_are_in_M} that $\mathfrak R(\mathcal E) \subseteq M$, which yields $\mathfrak R(\mathcal E) \cap \Lambda^+ = \mathfrak R(\mathcal E) \cap \Gamma(O)$ by Proposition~\ref{prop_Gamma(O)}.
Now Proposition~\ref{prop_partial_mu} and Corollary~\ref{cor_partial_mu} imply that every $\mu \in \mathfrak R(\mathcal E) \cap \Lambda^+$ is the weight of a standard $B$-root subgroup on~$X$.
It remains to apply Proposition~\ref{prop_weight_of_HRS}.
\end{proof}

\begin{example}[$G$-root subgroups on affine horospherical $G$-varieties]
Take $G = \SL_2 \times \KK^\times$ and let $\alpha$ (resp. $\chi$) be the unique positive root of $\SL_2$ (resp. a basis character of $\KK^\times$), so that $\Lambda^+ = \ZZ_{\ge0} \frac{\alpha}2 \oplus \ZZ\chi$ and $\mathfrak X(G) = \ZZ\chi$.
Suppose that $X$ is the affine horospherical $G$-variety with $\Gamma = \ZZ_{\ge0} \lbrace a\alpha + \chi, b\alpha - \chi\rbrace$ for fixed $a,b \in \ZZ_{>0}$.
Then easy calculations of the set $\mathfrak R(\mathcal E) \cap \mathfrak X(G)$ show that there are no $G$-root subgroups on~$X$ if none of the numbers $a,b$ is divisible by the other, there are exactly two such subgroups when $a=b=1$, and such subgroup is unique otherwise.
We note that there are infinitely many horizontal $B$-root subgroups on~$X$ regardless of the values of~$a,b$.
\end{example}

\begin{example}[A horizontal $B$-root subgroup on an affine horospherical variety of a simple group]
Take $G = \SL_n$ ($n \ge 3$) and choose $B$ (resp.~$T$) to be the subgroup of all upper-triangular (resp. diagonal) matrices in~$G$.
Consider the $G$-module $W = \Mat_{n\times n}(\KK) \oplus \KK^n$ with the action given by $(g, (A,v)) \mapsto (gAg^{-1}, gv)$ and let $X$ be the closure in~$W$ of the $G$-orbit of the pair of highest-weight vectors
\[
(\begin{pmatrix}
0 & \cdots & 0 & 1 \\
0 & \cdots & 0 & 0 \\
\vdots & \iddots & \vdots & \vdots \\
0 & \cdots & 0 & 0
\end{pmatrix},
\begin{pmatrix}
1 \\ 0 \\ \vdots \\ 0
\end{pmatrix}).
\]
In other words, $X = \lbrace (A,v) \in W \mid \tr(A) = 0, \ \rk(A\,|\, v) \le 1 \rbrace$.
In view of Remark~\ref{rem_AHV_construction}, $X$~is an affine horospherical $G$-variety and $\Gamma = \ZZ_{\ge0}\lbrace \omega_1 +\nobreak \omega_{n-1}, \omega_{n-1}\rbrace$ where $\omega_i \in \mathfrak X(T)$ is the $i$th fundamental weight of~$G$, so that $\omega_i(t) = t_1\ldots t_i$ for all $t = \diag(t_1,\ldots,t_n) \in T$.

For $i,j=1,\ldots,n$ let $a_{ij}$ (resp.~$x_i$) denote the restriction to~$X$ of the $(ij)$th coordinate function on $\Mat_{n\times n}(\KK)$ (resp. $i$th coordinate function on~$\KK^n$).
Then the algebra~$\KK[X]^U$ is freely generated by the two functions $a_{n1}$ and $x_n$ of weights $\omega_1 + \omega_{n-1}$ and $\omega_{n-1}$, respectively.

The weight $\mu=\omega_1$ is a Demazure root of the cone~$\mathcal E$.
Applying formula~(\ref{eqn_horo_LND}),
we see that the standard LND $\partial_{\mu}$ on~$\KK[X]$ annihilates all functions $a_{ij}$ and sends $x_i$ to $a_{i1}$ for all~$i=1,\ldots,n$.
This shows that the
corresponding $B$-root subgroup $H$ on~$X$ acts as
\[
(s,(A,v)) \mapsto (A,v+sA_1)
\]
where $A_1$ is the first column of the matrix $A$.
This subgroup moves the $G$-stable prime divisor $\{v=0\}$ on~$X$.
\end{example}

The following theorem confirms Conjecture~\ref{conj_G-stable} for affine horospherical $G$-varieties.

\begin{theorem} \label{thm_moved_divisors}
For every $D \in \mathcal D^G$, there is a $B$-root subgroup on $X$ that moves~$D$.
\end{theorem}

\begin{proof}
Since $\mathcal T = \lbrace 0 \rbrace$, every $D \in \mathcal D^G$ satisfies $\varkappa(D) \in \mathcal T^\perp$, which implies the assertion by Proposition~\ref{prop_moving_D}.
\end{proof}

\begin{remark}
It is shown in~\cite{GS} that, for an affine horospherical $G$-variety $X$ satisfying $\Gamma \cap (-\Gamma) = \lbrace 0 \rbrace$ (i.e. the cone $\mathcal G$ is strictly convex), the subgroup in $\Aut(X)$ generated by all $\GG_a$-subgroups acts on the regular locus of $X$ transitively.
This implies that every smooth point on a $G$-stable prime divisor in~$X$ can be moved to a point in the open $G$-orbit by an appropriate sequence of (not necessary root) $\GG_a$-subgroups.
The same transitivity property for an arbitrary affine spherical $G$-variety $X$ with strictly convex cone $\mathcal{G}$ is an open problem.
\end{remark}


\section{Reductive groups of semisimple rank one acting on toric varieties}
\label{sec6}

Throughout this section, $G$ is a connected reductive linear algebraic group of semisimple rank one.
Replacing $G$ by a finite covering, we assume that $G=\SL_2\times S$ where $S$ is an algebraic torus.
Let $T_0, U \subseteq \SL_2$ be a maximal torus and a maximal unipotent subgroup normalized by~$T_0$, respectively.
Then $T = T_0 \times S$ and $B = TU$ are a maximal torus and a Borel subgroup of~$G$, respectively.
Let $\alpha \in \mathfrak X(T)$ be the unique positive root of~$G$ with respect to~$B$ and let $\alpha^\vee \in \Hom(\mathfrak X(T),\ZZ)$ be the corresponding dual root, so that $\langle \alpha^\vee, \alpha \rangle = 2$ and $\langle \alpha^\vee, \chi \rangle = 0$ for all $\chi \in \mathfrak X(S)$.
As before, $X$ denotes an affine spherical $G$-variety, and we retain all the notation of~\S\,\ref{subsec_ASV_notation}.

Our main goal in this section is to obtain a complete description of the $B$-root subgroups on~$X$ in the case where $X$ is toric as a $T$-variety.
To this end, we first obtain a complete description of the $T$-root subgroups on~$X$ in terms of the weight monoid~$\Gamma$ and then determine which of the $T$-root subgroups are in fact $B$-root subgroups on~$X$.

\subsection{A criterion for the existence of an open \texorpdfstring{$T$}{T}-orbit}

For every $\lambda \in \Lambda^+$, we put $d_\lambda = \langle \alpha^\vee, \lambda \rangle$ for short.

\begin{proposition} \label{prop_criterion_toric}
The following conditions are equivalent.
\begin{enumerate}[label=\textup{(\arabic*)},ref=\textup{\arabic*}]
\item \label{}
$X$ is toric as a $T$-variety.

\item
$\alpha \notin \QQ\Gamma$.
\end{enumerate}
\end{proposition}

\begin{proof}
It follows from the representation theory of~$\SL_2$ that
for every $\lambda \in \Gamma$ there is a decomposition $\KK[X]_\lambda = \bigoplus \limits_{i=0}^{d_{\lambda}} \KK[X]_{\lambda,i}$ where $\KK[X]_{\lambda, i} \subseteq \KK[X]_\lambda$ is a one-dimensional $T$-submodule of weight~$\lambda - i\alpha$.
Then we have a $T$-module decomposition
\begin{equation} \label{eqn_toric_decomp}
\KK[X] = \bigoplus_{\lambda \in \Gamma} \bigoplus_{i=0}^{d_\lambda} \KK[X]_{\lambda, i}.
\end{equation}
Recall from Theorem~\ref{thm_VK} that $X$ is toric as a $T$-variety if and only if $\KK[X]$ is a multiplicity-free $T$-module.

If $\alpha \in \QQ\Gamma$ then $k\alpha = \lambda - \mu$ for some $k \in \ZZ_{>0}$ and $\lambda, \mu \in \Gamma$.
Then $\KK[X]_{\lambda,k} \simeq \KK[X]_{\mu,0}$ as $T$-modules and thus the $T$-module $\KK[X]$ is not multiplicity free.

Conversely, if $\KK[X]$ is not multiplicity free as a $T$-module then there are $\lambda, \mu \in \Gamma$ and $k,l\in \ZZ$ such that $\lambda \ne \mu$, $0 \le k \le d_\lambda$, $0 \le l \le d_\mu$, and $\KK[X]_{\lambda,k} \simeq \KK[X]_{\mu,l}$ as $T$-modules.
It follows that $\lambda - k\alpha = \mu - l\alpha$, which implies $\alpha \in \QQ\Gamma$ as $k \ne l$.
\end{proof}

\begin{corollary} \label{crl_SL2_horo}
If $X$ is toric as a $T$-variety then $X$ is horospherical as a $G$-variety.
\end{corollary}

\begin{proof}
Thanks to Theorem~\ref{thm_horospherical}, it suffices to prove that $\KK[X]_\lambda \cdot \KK[X]_\mu \subseteq \KK[X]_{\lambda+\mu}$ for all $\lambda, \mu \in \Gamma$.
The multiplication of $\KK[X]_\lambda$ and $\KK[X]_\mu$ induces a $G$-module homomorphism $\varphi \colon \KK[X]_\lambda \otimes \KK[X]_\mu \to \KK[X]$, and it suffices to show that $\Im \varphi = \KK[X]_{\lambda +\mu}$.
As $f_\lambda \cdot f_\mu = f_{\lambda+\mu}$, we have $\KK[X]_{\lambda+\mu} \subseteq \Im \varphi$.
If $\KK[X]_{\nu} \subseteq \Im \varphi$ for some $\nu \in \Gamma \setminus \lbrace \lambda + \mu \rbrace$ then $\lambda+\mu - \nu$ is a positive multiple of~$\alpha$, which is impossible by Proposition~\ref{prop_criterion_toric}.
\end{proof}

\subsection{Auxiliary results}
\label{subsec_aux_res}

In what follows we assume that $X$ is toric as a $T$-variety and the derived subgroup $(G,G) = \SL_2$ acts nontrivially on~$X$, which is equivalent to $\Gamma \not\subseteq \mathfrak X(S)$.
Now $X$ can be regarded both as a spherical $G$-variety and as a toric $T$-variety.
Each of these points of view involves its own combinatorial data attached to~$X$, and in this subsection we clarify relations between them.

Let $w$ be the nontrivial element of the Weyl group $N(T)/T$, where $N(T)$ is the normalizer of~$T$ in~$G$.
Then $w$ naturally acts on $\mathfrak X(T)$ in such a way that $w(\alpha) = -\alpha$ and $w(\chi) = \chi$ for all $\chi \in \mathfrak X(S)$.

Put $\overline M = M \oplus \ZZ\alpha \subseteq \mathfrak X(T)$, $\overline N = \Hom_\ZZ(\overline M, \ZZ)$, $\overline M_\QQ = \overline M \otimes_\ZZ \QQ$, $\overline N_\QQ = \overline N \otimes_\ZZ \QQ$ and extend the pairing $\langle \cdot\,, \cdot \rangle$ to a bilinear map $\overline M_\QQ \times \overline N_\QQ \to \QQ$.
For every cone $\mathcal C \subseteq M_\QQ$, let $\overline{\mathcal C}$ denote the cone in $\overline M_\QQ$ generated by $\mathcal C \cup w(\mathcal C)$.
It follows from~(\ref{eqn_toric_decomp}) that the weight monoid of $X$ as a toric $T$-variety is $\overline \Gamma := \overline{\mathcal G} \cap \overline M$.
Let $\overline{\mathcal E} \subseteq \overline N_\QQ$ be the cone dual to~$\overline{\mathcal G}$.

To describe the $T$-root subgroups on~$X$, we need to know the facets of $\overline{\mathcal G}$ or, equivalently, the set $\overline{\mathcal E}^1$.
It is easy to see that $\mathcal G, w(\mathcal G)$ are facets of $\overline{\mathcal G}$, and we let $\delta, \delta' \in \overline{\mathcal E}^1$ be the corresponding elements.
Observe that
\begin{gather}
\langle \delta, M \rangle = 0, \;\; \langle \delta, \alpha \rangle = -1; \label{eqn_delta} \\
\langle \delta', w(M) \rangle =0, \;\; \langle \delta', \alpha \rangle = 1. \label{eqn_delta'}
\end{gather}
(The precise values of $\langle \delta, \alpha \rangle$ and $\langle \delta', \alpha \rangle$ are implied by the primitivity of $\delta,\delta'$ in~$\overline N$.)

\begin{lemma} \label{lemma_delta'}
One has $\iota(\delta') = \iota(\alpha^\vee)$.
\end{lemma}

\begin{proof}
Take any $\lambda \in M$.
Since $w(\lambda) = \lambda - \langle \alpha^\vee, \lambda \rangle \alpha$, by~(\ref{eqn_delta'}) one has
$\langle \delta', \lambda - \langle \alpha^\vee,\lambda \rangle \alpha \rangle = 0$ and therefore $\langle \delta', \lambda \rangle = \langle \delta', \alpha \rangle \langle \alpha^\vee, \lambda \rangle = \langle \alpha^\vee, \lambda \rangle$ as required.
\end{proof}

Let $\mathcal Q \subseteq \overline M_\QQ$ be the subspace spanned by $\overline M_\QQ \cap \mathfrak X(S)$; one has $M_\QQ \not\subseteq \mathcal Q$ by our assumptions.
The next result provides an explicit relation between the facets of $\mathcal G$ and $\overline{\mathcal G}$ and also between the sets $\mathcal E^1$ and $\overline{\mathcal E}^1$.

\begin{lemma} \label{lemma_facets}
The following assertions hold.
\begin{enumerate}[label=\textup{(\alph*)},ref=\textup{\alph*}]
\item \label{lemma_facets_a}
If $\mathcal F$ is a facet of~$\mathcal G$ contained in~$\mathcal Q$ and $\rho \in \mathcal E^1$ corresponds to $\mathcal F$ then $\iota(\delta')$ is a positive multiple of~$\rho$.
In particular, $\mathcal F$ is unique if exists.

\item \label{lemma_facets_b}
The map $\mathcal F \mapsto \overline{\mathcal F}$ yields a bijection
\[
\lbrace \text{facets of $\mathcal G$ not contained in~$\mathcal Q$} \rbrace \to
\lbrace \text{facets of $\overline{\mathcal G}$ different from~$\mathcal G$ and~$w(\mathcal G)$} \rbrace.
\]
Moreover, if $\rho \in \mathcal E^1$ \textup(resp. $\overline \rho \in \overline{\mathcal E}^1$\textup) corresponds to $\mathcal F$ \textup(resp.~$\overline{\mathcal F}$\textup) then $\overline \rho$ is the extension of~$\rho$ to~$\overline N$ defined by $\langle \overline \rho, \alpha \rangle = 0$.
\end{enumerate}
\end{lemma}

\begin{proof}
Choose a finite generating set $\mathrm E$ for $\Gamma$.

(\ref{lemma_facets_a})
Since $M_\QQ \not\subseteq \mathcal Q$, it follows that $\rho$ and $\iota(\delta')$ have the same kernel equal to~$\QQ \mathcal F$, therefore they are proportional.
It remains to observe that both $\rho$ and $\iota(\delta')$ take positive values on any element in~$\mathrm E \setminus \mathcal Q$.

(\ref{lemma_facets_b})
Let $\mathcal F$ be a facet of $\mathcal G$ not contained in~$\mathcal Q$, let $\rho \in \mathcal E^1$ be the corresponding element, and extend $\rho$ to an element $\overline \rho \in \overline N$ by setting $\langle \overline \rho, \alpha \rangle = 0$.
Then $\overline{\mathcal F} = \overline{\mathcal G} \cap \Ker \overline \rho$, and so $\overline{\mathcal F}$ is a face of~$\overline{\mathcal G}$.
Clearly, $\overline{\mathcal F}$ has codimension one in~$\overline{\mathcal G}$ and $\overline{\mathcal F} \notin \lbrace \mathcal G, w(\mathcal G) \rbrace$.

Conversely, take any $\overline \rho \in \overline{\mathcal E}^1 \setminus \lbrace \delta, \delta' \rbrace$ and let $\mathcal F_{\overline \rho}$ be the corresponding facet of~$\overline{\mathcal G}$.
If $\langle \overline \rho, \alpha \rangle < 0$ then $\langle \overline \rho, \lambda \rangle > 0$ for all $\lambda \in w(\mathrm E) \setminus \mathrm E$, hence $\mathcal F_{\overline \rho}$ is contained in~$\mathcal G$, hence $\mathcal F_{\overline \rho} = \mathcal G$ and $\overline \rho = \delta$, which is excluded.
If $\langle \overline \rho, \alpha \rangle > 0$ then we similarly obtain $\mathcal F_{\overline \rho} = w(\mathcal G)$ and $\overline \rho = \delta'$, which is also excluded.
Thus $\langle \overline \rho, \alpha \rangle = 0$ and in particular $\mathcal F_{\overline \rho}$ is $w$-stable.
Let $\rho$ be the restriction of $\overline \rho$ to~$N$.
Then $\rho \in \mathcal E$ and therefore $\mathcal F = \mathcal G \cap \Ker \rho$ is a face of~$\mathcal G$.
Now observe that $\mathcal F = \mathcal G \cap \mathcal F_{\overline \rho}$, hence $\mathcal F$ is generated by the set $\mathrm E \cap \mathcal F_{\overline \rho}$.
Since $\mathcal F_{\overline \rho}$ is $w$-stable, it follows that $\mathcal F_{\overline \rho} = \overline{\mathcal F}$.
By dimension reasons, $\mathcal F$ is a facet of $\mathcal G$ not contained in~$\mathcal Q$.
\end{proof}

Combining (\ref{eqn_delta}), (\ref{eqn_delta'}) with Lemma~\ref{lemma_facets}(\ref{lemma_facets_b}) we obtain

\begin{corollary} \label{crl_bar_rho}
Suppose $\overline \rho \in \overline{\mathcal E}^1$; then
\[
\langle \overline \rho, \alpha \rangle =
\begin{cases}
-1 & \text{if} \ \, \overline\rho = \delta;\\
1 & \text{if} \ \, \overline\rho = \delta';\\
0 & \text{otherwise}.
\end{cases}
\]
In particular, $\alpha \in \mathfrak R_{\delta}(\overline{\mathcal E})$ and $-\alpha \in \mathfrak R_{\delta'}(\overline{\mathcal E})$.
\end{corollary}

The next proposition will be a key ingredient for converting the description of $T$-root subgroups on~$X$ in terms of~$\overline \Gamma$ into a description of $B$-root subgroups on~$X$ in terms of~$\Gamma$.

\begin{lemma} \label{lemma_DR_in_M}
The following assertions hold.
\begin{enumerate}[label=\textup{(\alph*)},ref=\textup{\alph*}]
\item \label{lemma_DR_in_M_a}
$\mathfrak R(\mathcal E) \cap \Lambda^+ \subseteq \mathfrak R(\overline{\mathcal E})$.
\item \label{lemma_DR_in_M_b}
Suppose $e \in \mathfrak R(\overline{\mathcal E})$ and $\delta_e \in \overline{\mathcal E}^1$ is the corresponding element.
Then $e \in \mathfrak R(\mathcal E) \cap \Lambda^+$ if and only if $\langle \delta, e \rangle = \langle \delta_e, \alpha \rangle = 0$.
\end{enumerate}
\end{lemma}

\begin{proof}
Suppose $e \in \mathfrak R(\mathcal E) \cap \Lambda^+$ and let $\rho \in \mathcal E^1$ be the corresponding element, so that $\langle \rho, e \rangle = -1$.
Remark~\ref{rem_DR_are_in_M} yields $e \in M$, therefore $\langle \delta, e \rangle = 0$ by~(\ref{eqn_delta}).
Next, $\langle \delta', e \rangle \ge 0$ by Lemma~\ref{lemma_delta'}.
Let $\mathcal F$ be the facet of~$\mathcal G$ corresponding to~$\rho$.
Applying Lemma~\ref{lemma_facets}(\ref{lemma_facets_a}) we find that $\mathcal F \not\subseteq \mathcal Q$.
Then Lemma~\ref{lemma_facets}(\ref{lemma_facets_b}) implies that $\rho$ extends to an element $\overline\rho \in \overline{\mathcal E}^1$ with $\langle \overline\rho, \alpha \rangle = 0$ and every $\overline\rho' \in \overline{\mathcal E}^1 \setminus \lbrace \delta, \delta', \overline \rho \rbrace$ is an extension of an element in $\mathcal E^1 \setminus \lbrace \rho \rbrace$, which implies $\langle \overline\rho', e \rangle \ge 0$.
It follows that $e \in \mathfrak R(\overline{\mathcal E})$, and we have proved~(\ref{lemma_DR_in_M_a}) and the `only if' part of~(\ref{lemma_DR_in_M_b}).

It remains to prove the `if' part of~(\ref{lemma_DR_in_M_b}).
Let $e,\delta_e$ be as in the hypothesis and suppose $\langle \delta, e \rangle = \langle \delta_e, \alpha \rangle = 0$.
Then $\delta_e \notin \lbrace \delta,\delta' \rbrace$ in view of~(\ref{eqn_delta}) and~(\ref{eqn_delta'}).
As $\langle \delta, e \rangle = 0$, one has $e \in M$.
Let $\mathcal F_e$ be the facet of $\overline{\mathcal G}$ corresponding to~$\delta_e$ and let $\rho_e$ be the restriction of $\delta_e$ to~$M$, so that $\langle \rho_e, e \rangle = -1$.
By Lemma~\ref{lemma_facets}(\ref{lemma_facets_b}), the condition $\langle \delta_e, \alpha \rangle = 0$ implies $\rho_e \in \mathcal E^1$ and $\mathcal F_e = \overline{\mathcal F}$ for a facet $\mathcal F$ of~$\mathcal G$ not contained in~$\mathcal Q$.
Now take any $\rho \in \mathcal E^1 \setminus \lbrace \rho_e \rbrace$ and let $\mathcal F'$ be the corresponding facet of~$\mathcal G$.
If $\mathcal F' \subseteq \mathcal Q$ then Lemma~\ref{lemma_facets}(\ref{lemma_facets_a}) yields $\langle \rho, e \rangle \ge 0$.
If $\mathcal F' \not\subseteq \mathcal Q$ then, by Lemma~\ref{lemma_facets}(\ref{lemma_facets_b}), $\rho$ extends to an element of $\overline{\mathcal E}^1 \setminus \lbrace \delta_e \rbrace$, hence again $\langle \rho, e \rangle \ge 0$, and
we thus obtain $e \in \mathfrak R(\mathcal E)$.
Lemma~\ref{lemma_delta'} implies $\langle \alpha^\vee, e \rangle \ge 0$, and so $e \in \Lambda^+$.
\end{proof}

\subsection{Main results}

Retain the notation and assumptions of \S\,\ref{subsec_aux_res}.
Recall from~\S\,\ref{subsec_RS_on_ATV} that there is a bijection $\overline \rho \mapsto D_{\overline \rho}$ between $\overline{\mathcal E}^1$ and the $T$-stable prime divisors on~$X$.

\begin{proposition} \label{prop_divisors}
Given $\overline \rho \in \overline{\mathcal E}^1$, the divisor
$D_{\overline \rho}$ is $B$-stable \textup(resp. $B^-$-stable, $G$-stable\textup) if and only if $\overline \rho \in \overline{\mathcal E}^1 \setminus \lbrace \delta \rbrace$ \textup(resp. $\overline \rho \in \overline{\mathcal E}^1 \setminus \lbrace \delta' \rbrace$, $\overline \rho \in \overline{\mathcal E}^1 \setminus \lbrace \delta, \delta' \rbrace$\textup).
In particular, $D_{\delta'}$ is the unique color in~$X$.
\end{proposition}

\begin{proof}
Clearly, $D_{\overline \rho}$ is $B$-stable if and only if it is $U$-stable.
Since $U$ is a $T$-root subgroup on~$X$ of weight~$\alpha$ and $\alpha \in \mathfrak R_{\delta}(\overline{\mathcal E})$ by
Corollary~\ref{crl_bar_rho}, it follows from Theorem~\ref{thm_toric_equivalence} that $D_{\overline \rho}$ is $U$-stable if and only if $\overline \rho \ne \delta$.
Similarly, $D_{\overline \rho}$ is $B^-$-stable if and only if $\overline \rho \ne \delta'$.
It remains to notice that $D_{\overline \rho}$ is $G$-stable if and only if it is simultaneously $B$-stable and $B^-$-stable.
\end{proof}

The following theorem provides a complete combinatorial description of all $B$-root subgroups on~$X$.
Part~(\ref{thm_SL2_toric_c}) of this theorem can be deduced from Corollary~\ref{crl_SL2_horo} and Proposition~\ref{prop_horo_weights}(\ref{prop_horo_weights_b}), but we provide a direct proof based on Lemma~\ref{lemma_DR_in_M}.

\begin{theorem} \label{thm_SL2_toric}
The following assertions hold.
\begin{enumerate}[label=\textup{(\alph*)},ref=\textup{\alph*}]
\item \label{thm_SL2_toric_a}
Every $B$-root subgroup on~$X$ is uniquely determined by its weight.

\item \label{thm_SL2_toric_b}
Every vertical $B$-root subgroup on~$X$ is equivalent to~$U$ and the set of weights of such subgroups is $\mathfrak R_{\delta}(\overline{\mathcal E})$.

\item \label{thm_SL2_toric_c}
Every horizontal $B$-root subgroup on~$X$ is standard and the set of weights of such subgroups is $\mathfrak R(\mathcal E) \cap \Lambda^+$.
\end{enumerate}
\end{theorem}

\begin{proof}
Since every $B$-root subgroup on~$X$ is a $T$-root subgroup and $X$ is toric as a $T$-variety, part~(\ref{thm_SL2_toric_a}) is implied by Theorem~\ref{thm_T-root_subgroups}(\ref{thm_T-root_subgroups_b}).
In what follows we simultaneously prove~(\ref{thm_SL2_toric_b}) and~(\ref{thm_SL2_toric_c}).

Take any $e \in \mathfrak R(\overline{\mathcal E})$, consider the corresponding element $\delta_e \in \overline{\mathcal E}^1$, and let $H_e$ be the $T$-root subgroup on~$X$ of weight~$e$.
Clearly, $H_e$ is a $B$-root subgroup on~$X$ if and only if $H_e$ commutes with $U$, which in turn is a $T$-root subgroup on~$X$ of weight~$\alpha \in \mathfrak R_{\delta}(\overline{\mathcal E})$ (see
Corollary~\ref{crl_bar_rho}).
Thanks to~\cite[Lemma~3.2]{AR}, $H_e$ and $U$ commute if and only if one of the following two cases occurs.

\textit{Case}~1: $e \in \mathfrak R_{\delta}(\overline{\mathcal E})$.
By Theorem~\ref{thm_toric_equivalence}, this means that $H_e$ is equivalent to $U$, and so $H_e$ is vertical.

\textit{Case}~2: $\langle \delta, e \rangle = \langle \delta_e, \alpha \rangle = 0$.
By Lemma~\ref{lemma_DR_in_M}(\ref{lemma_DR_in_M_b}), this happens if and only if $e \in \mathfrak R(\mathcal E) \cap \Lambda^+$.
Since $\delta_e \ne \delta$ by~(\ref{eqn_delta}), we find from Proposition~\ref{prop_divisors} that the divisor $D_{\delta_e}$ moved by~$H_e$ is in fact $B$-stable, hence $H_e$ is horizontal.
Moreover, from~(\ref{eqn_T-norm_LND}) and $\langle \delta_e, \alpha \rangle = 0$ we conclude that $H_e$ is standard.

As $\mathfrak R(\mathcal E) \cap \Lambda^+ \subseteq \mathfrak R(\overline{\mathcal E})$ by Lemma~\ref{lemma_DR_in_M}(\ref{lemma_DR_in_M_a}), the proof of~(\ref{thm_SL2_toric_b}) and~(\ref{thm_SL2_toric_c}) is completed.
\end{proof}

\begin{remark}
A $B$-root subgroup $H$ on $X$ is $G$-normalized if and only if it commutes with~$U^{-}$, which is a $T$-root subgroup of weight~$-\alpha \in \mathfrak R_{\delta'}(\overline{\mathcal E})$.
Applying again~\cite[Lemma~3.2]{AR} and using Theorem~\ref{thm_SL2_toric} we find that $H$ is a $G$-root subgroup if and only if $\chi_H \in \mathfrak R(\mathcal E) \cap \Lambda^+ \cap \Ker \delta'$.
By Lemma~\ref{lemma_delta'}, the latter set equals $\mathfrak R(\mathcal E) \cap \mathfrak X(G)$, which yields a
direct proof of Proposition~\ref{prop_horo_weights}(\ref{prop_horo_weights_a}) for our~$X$.
\end{remark}

As follows from Theorem~\ref{thm_SL2_toric}(\ref{thm_SL2_toric_b}), there is only one equivalence class of vertical $B$-root subgroups on~$X$: all such subgroups are equivalent to~$U$.
Our last goal is to determine the equivalence classes of horizontal $B$-root subgroups on~$X$.
Before stating and proving the result (see Theorem~\ref{thm_equiv_classes}), we need the following

\begin{lemma} \label{lemma_nonempty}
For every $\overline\rho \in \overline{\mathcal E}^1 \setminus \lbrace \delta, \delta' \rbrace$, the set $\mathfrak R_{\overline\rho}(\overline{\mathcal E}) \cap \mathfrak R(\mathcal E) \cap \Lambda^+$ is nonempty.
\end{lemma}

\begin{proof}
Fix any $\overline \rho \in \overline{\mathcal E}^1 \setminus \lbrace \delta, \delta' \rbrace$ and $e' \in \mathfrak R_{\overline\rho}(\overline{\mathcal E})$.
Then $\langle {\overline\rho}, e' \rangle = -1$ and $\langle \overline\rho', e' \rangle \ge 0$ for all $\overline\rho' \in \overline{\mathcal E}^1 \setminus \lbrace \overline\rho \rbrace$.
Put $e = e' + q \alpha$ where $q = \langle \delta, e' \rangle \ge 0$.
Corollary~\ref{crl_bar_rho} yields $\langle \overline\rho, e \rangle = -1$, $\langle \overline\rho', e \rangle \ge 0$ for all $\overline \rho' \in \overline{\mathcal E}^1 \setminus \lbrace \delta, \delta', \overline \rho \rbrace$, $\langle \delta, e \rangle = 0$, and $\langle \delta', e \rangle = \langle \delta', e' \rangle + q \ge 0$, which implies $e \in \mathfrak R_{\overline\rho}(\overline{\mathcal E})$.
As $\langle \delta, e \rangle = 0$ and $\langle \overline \rho, \alpha \rangle = 0$, Lemma~\ref{lemma_DR_in_M}(\ref{lemma_DR_in_M_b}) yields $e \in \mathfrak R(\mathcal E) \cap \Lambda^+$.
\end{proof}

\begin{theorem} \label{thm_equiv_classes}
The equivalence classes of horizontal $B$-root subgroups on~$X$ are in bijection with the $G$-stable prime divisors on~$X$.
More precisely, under this bijection every $G$-stable prime divisor $D \subseteq X$ corresponds to all $B$-root subgroups that move~$D$.
\end{theorem}

\begin{proof}
Let $H$ be a horizontal $B$-root subgroup on~$X$ and let $\overline\rho \in \overline{\mathcal E}^1$ be such that $\chi_H \in \mathfrak R_{\overline\rho}(\overline{\mathcal E})$, so that $H$ moves the divisor~$D_{\overline\rho}$.
By Theorem~\ref{thm_SL2_toric}(\ref{thm_SL2_toric_c}) and Lemma~\ref{lemma_DR_in_M}(\ref{lemma_DR_in_M_b}), we have $\langle \overline \rho, \alpha \rangle = 0$.
Then Corollary~\ref{crl_bar_rho} yields $\overline\rho \notin \lbrace \delta, \delta' \rbrace$, and so $D_{\overline\rho}$ is $G$-stable by Proposition~\ref{prop_divisors}.
On the other hand, thanks to Proposition~\ref{prop_divisors}, Theorem~\ref{thm_SL2_toric}(\ref{thm_SL2_toric_c}), and Lemma~\ref{lemma_nonempty}, for every $G$-stable prime divisor~$D$ on~$X$ there exists a horizontal $B$-root subgroup that moves~$D$.
The rest is implied by Theorem~\ref{thm_toric_equivalence}.
\end{proof}

\begin{example}[An illustration of the main results]
Retain the situation and notation of Example~\ref{ex_moved_colors} and take $X' = D_3 = \lbrace \det = 0 \rbrace$.
Then $X'$ is toric as a $T$-variety and $\Gamma(X') = \ZZ_{\ge0} \lbrace \omega + \chi_1, \omega + \chi_2 \rbrace$.
There are two $G$-stable prime divisors $D'_1 = \lbrace x_{11}=x_{21} =0 \rbrace$, $D'_2 = \lbrace x_{12} = x_{22} = 0 \rbrace$ and one color $D' = \lbrace x_{21} = x_{22} = 0 \rbrace$ in~$X'$.
By Corollary~\ref{crl_no_moved_colors}, there is no $B$-root subgroup on $X'$ moving~$D'$.
For $i=1,2$, the set of weights of (the equivalence class of) $B$-root subgroups on $X'$ moving $D'_i$ is $\chi_j - \chi_i + \ZZ_{\ge0}(\omega + \chi_j)$ (where $j = 3-i$); among these subgroups there is a unique $G$-root subgroup, which is induced by the action of $H_i$ and has the weight $\chi_j - \chi_i$.
The set of weights of vertical $B$-root subgroups on~$X'$ is $\alpha + \Gamma(X')$, and so all such subgroups are replicas of~$U$.
\end{example}

\begin{remark}
The general strategy in this section is to take an affine spherical $G$-variety $X$ and impose the restriction that $T$ acts on~$X$ with an open orbit.
We could also start with an affine toric $T$-variety $X$ and impose the restriction that the $T$-action on $X$ extends to a $G$-action; this situation is discussed in~\cite[\S\,2.2]{AL}.
In this way one can obtain an alternative proof of Theorem~\ref{thm_equiv_classes}.
\end{remark}

\begin{remark}
It is an interesting problem to study $B$-root subgroups on affine spherical varieties of a connected reductive group $G$ of semisimple rank one such that the maximal torus $T$ in $G$ has no open orbit in $X$. In this case $T$ acts on $X$ with complexity one and a description of $B$-root subgroups can be obtained using the results of~\cite{Ma}.
\end{remark}


\end{document}